\documentclass[12pt]{amsart}
\usepackage{amsbsy,amssymb,amscd,amsfonts,latexsym,amstext,delarray,
amsmath,graphicx, mathrsfs, bbold} 
\input xypic

\usepackage{tikz-cd}

\usepackage{color}

\newtheorem{thm}{Theorem}[section]
\newtheorem{prop}[thm]{Proposition}
\newtheorem{cor}[thm]{Corollary}
\newtheorem{lem}[thm]{Lemma}

\newtheorem{defn}[thm]{Definition}
\newtheorem{rem}[thm]{Remark}

\numberwithin{equation}{section}

\def\bA{{\mathbb A}}

\def\bD{{\mathbb D}}
\def\bE{{\mathbb E}}

\def\bP{{\mathbb P}}

\def\bT{{\mathbb T}}

\def\bV{{\mathbb V}}

\def\Q{{\mathbb Q}}

\def\Z{{\mathbb Z}}

\renewcommand{\epsilon}{{\varepsilon}}
\renewcommand{\phi}{{\varphi}}

\def\cA{{\mathcal A}}
\def\cB{{\mathcal B}}
\def\cC{{\mathcal C}}
\def\cD{{\mathcal D}}
\def\cE{{\mathcal E}}
\def\cF{{\mathcal F}}
\def\cG{{\mathcal G}}

\def\cL{{\mathcal L}}
\def\cM{{\mathcal M}}
\def\cN{{\mathcal N}}
\def\cO{{\mathcal O}}
\def\cP{{\mathcal P}}

\def\cS{{\mathcal S}}

\def\cW{{\mathcal W}}

\def\cZ{{\mathcal Z}}

\def\Coker{{\rm Coker}}

\def\Hom{{\rm Hom}}
\def\id{{\rm id}}

\def\Ker{{\rm Ker}}

\def\Mod{{\rm Mod}}

\def\rank{{\rm rank}}

\def\Spec{{\rm Spec}}

\def\cof{{\textbf{cof}}}

\def\Cat{{\textbf{Cat}}}
\def\PrSh{{\textbf{PrSh}}}
\def\comp{{\textbf{comp}}}
\def\Ho{{\textbf{Ho}}}
\def\Ob{{\text{Ob}}}
\def\Var{{\textbf{Var}}}
\def\Sch{{\textbf{Sch}}}
\def\Sm{{\textbf{Sm}}}
\def\SH{{\textbf{SH}}}
\def\DM{{\textbf{DM}}}

\def\dm{{\text{DM}}}
\def\Ex{{\text{Ex}}}
\def\cons{{\text{cons}}}
\def\sub{{\textbf{sub}}}
\def\Chow{{\text{Chow}}}

\newcommand\triplerightarrow{%
        \mathrel{\vcenter{\mathsurround0pt
                \ialign{##\crcr
                        \noalign{\nointerlineskip}$\rightarrow$\crcr
                        \noalign{\nointerlineskip}$\rightarrow$\crcr
                        \noalign{\nointerlineskip}$\rightarrow$\crcr
                }%
        }}%
}

\title[From Six Functors Formalisms to Derived Motivic Measures]{From Six Functors Formalisms to Derived Motivic Measures} 
\author{Joshua F.~Lieber}

\address{California Institute of Technology, Pasadena \\ USA}
\email{jlieber@caltech.edu}

\begin{document}
\maketitle

\begin{abstract}
    In this paper, we generally describe a method of taking an abstract six functors formalism in the sense of Khan or Cisinski-D\'{e}glise, and outputting a derived motivic measure in the sense of Campbell-Wolfson-Zakharevich.  In particular, we use this framework to define a lifting of the Gillet-Soul\'{e} motivic measure.
\end{abstract}

Conventions: Throughout this paper, $k$ will always refer to a perfect field, and $R$ will refer to an arbitrary commutative ring.  Furthermore, given a base scheme $S$, the category $\Var_S$ of varieties over $S$ will simply be the category of finite type separated schemes over $S$ (we do not require our varieties to be reduced).  Importantly, unless otherwise stated, our ambient $\infty$-cosmos \cite{RieVer} will be that of Kan complex-enriched categories, the so-called fibrant $S$-categories of To\"en-Vezzozi \cite{ToeVez2} .  This is because we will be working with commutative diagrams directly for much of the work, and it is easier to prove commutativity in a model with strict horizontal composition.  That said, we may invoke comparison between the K-theory of $S$-categories and that of Quasicategories developed in \cite{BGT} .

\section{Introduction}

This paper was born out of a desire to construct a lift of the Gillet-Soul\'e motivic measure.  While originally more limited in scope, its current iteration synthesizes a number of different recent developments in algebraic geometry and homotopy theory with an eye towards constructing meaningful derived motivic measures from six functors formalism.

One of the unusual features of this construction is that we are using something quite rare\textemdash namely, a six functors formalism (a highly structured object) to define a map of K-theory spectra (much less structure ultimately).  That said, neither object is particularly approachable in and of itself, and our hope is that each can be used to study the other.

We begin the paper with a brief exploration of the three main types of categories we will require for the constructions that follow.  These are Waldhausen categories, a now classical type of relative category used as a particularly general setting of K-theory; SW-categories, a more recently developed analogue of Waldhausen categories used to describe and categorify cutting and pasting data; and stable $\infty$-categories, a recent $\infty$-categorical enrichment of the notion of triangulated category which simplifies the definition considerably and makes many constructions (such as gluing and taking cones) functorial that are simply not on the level of triangulated categories.  Note that we will not discuss the formalisms of model categories or DG/spectral categories in this work because they are only briefly touched upon.  The motivated reader is referred to the phenomenal book by Hovey \cite{Hov} for model categories and the wonderful introduction by To\"en \cite{Toe} for DG categories (spectral categories are a focal point of the groundbreaking paper \cite{BGT}). 

We will then briefly look at how to extract K-theory from each of these types of categories.  The one novel aspect of this section will be to demonstrate how to meaningfully extend the definition of weakly W-exact functor \cite{Cam} from the original setting (in which one compares an SW-category to a Waldhausen category) to a comparison between SW-categories and pointed, finitely homotopy cocomplete $\infty$-categories (as before incarnated as fibrant $S$-categories):

\begin{thm}
Given a weakly W-exact functor $F:\cC\to\cA$, where $\cC$ is an SW-category and $\cA$ is a pointed finitely homotopy cocomplete $\infty$-category, it may be composed with the Yoneda embedding $\cA\to\cM(\cA)$ to obtain a weakly W-exact functor $\iota\circ F:\cC\to\cM(\cA)$ to the Waldhausen $\cM(\cA)$ which yields a map on K-theory $K(\cC)\to K(\cA)$.  This construction is functorial in exact functors both of SW-categories and of pointed finitely homotopy cocomplete $\infty$-categories.
\end{thm}

The interesting thing here is that in some sense, this is not even an extension of the definition, as given a particularly nice $\infty$-category $\cA$ (in a sense to be defined in a later section) that embeds into a simplicial model category, we may functorially assign it a Waldhausen category $\cM(\cA)$ such that $K(\cA)\simeq K(\cM(\cA))\simeq K(N^{hc}(\cA^{cf}))$, where the last term is the K-theory of the homotopy coherent nerve of the cofibrant-fibrant objects in $\cA$, which is the corresponding quasicategory under the homotopy coherent nerve functor.  It just so happens that in all cases we consider, the necessary niceness conditions will be met.

From here, we go on to describe the $\infty$-categorical six functors formalism outlined in Khan \cite{Khan}, alternately titled the theory of $(*, \#, \otimes)$-formalisms satisfying the Voevodsky criteria or the theory of motivic $\infty$-categories.  Upon introducing the basic setup, we go on to show that for a motivic $\infty$-category $\bD$ valued over $\cS$, where $\cS$ is some appropriate subcategory of schemes, given any excellent geometrically unibranch scheme $S\in\cS$ and a variety $f:X\to S$ over $S$, the assignment $X\mapsto f_*f^!(\mathbf1_S)$, where $\mathbf1_S$ is the tensor unit $\bD(S)$, upgrades to a weakly W-exact functor in the following way.

\begin{thm}
Suppose that $\bD$ satisfies one of the following two sets of conditions:
\begin{itemize}
    \item The four functors preserve constructible objects when given input a seperable morphism of finite type (note that compactness is trivially preserved by tensor)
    \item The six functors preserve constructible objects over Noetherian quasi-excellent schemes  of finite dimension with respect to morphisms of finite type (in other words, for any finite type morphism $f:X\to S$ with target Noetherian quasi-excellent of finite dimension, the four functors preserve constructible objects, while tensor and $\Hom$ preserve constructible objects over $S$)
\end{itemize}  
Then, given a scheme $S$ (assumed to be Noetherian quasi-excellent of finite dimension if $\bD$ satisfies the second condition in particular), there is a weakly W-exact functor \[M^c_{\bD(S)}:\Var_S\to\bD_{\cons}(S)\]
sending each variety (smooth or otherwise) $(X\overset{f}{\to}S)\in\Var_S$ to $M^c_{\bD(S)}(X):=f_*f^!\mathbf{1}_S$.
\end{thm}

This theorem forms the core of the paper, as it allows us to conclude the existence of a derived motivic measure:

\begin{cor}
Suppose $\bD$ sastisfies one of the two conditions of the above theorem.  Then, given a scheme $S$ (assumed to be Noetherian quasi-excellent of finite dimension if $\bD$ satisfies the second set of conditions), one obtains a map of K-theory spectra 
\[K(M^c_{\bD(S)}):K(\Var_S)\to K(\bD_{\cons}(S)).\]
\end{cor}

From here, we apply our construction to the motivic $\infty$-category of Beilinson motives $\DM_B$, and show that the construction restricts to a spectral lift of the Gillet-Soul\'e motivic measure, thus demonstrating our stated intent.

\begin{thm}
Considering a perfect base field $k$ and rational coefficients, the map 
\[
K(M^c_k):K(\Var_k)\to K(\DM_B^c(k))
\]
yields a derived lift of the Gillet-Soul\'e motivic measure.
\end{thm}

This paper concludes with a brief discussion of another approach to lifting the Gillet-Soul\'e motivic measure (which is almost surely equivalent), before mentioning some future directions this work might take.

\section{Waldhausen Categories, SW-Categories, Stable $\infty$-Categories and the K-Theory of Varieties}

\subsection{Relevant Categorical Definitions and Examples}

In this section, we will outline the different notions of $1$ or $\infty$-categories that are needed in this paper, as well as elucidate how to extract K-theory from each of them.

\begin{defn}
A \emph{Waldhausen category} $(\cC, \cof, \cW)$ consists of the data of a category $\cC$ with two distinguished classes of morphisms: \emph{cofibrations} $\cof$ (elements of which are denoted $\hookrightarrow$) and \emph{weak equivalences} $\cW$ (elements of which are denoted $\overset{\sim}{\to}$) which are required to satisfy the following axioms:
\begin{enumerate}
    \item All isomorphisms are cofibrations
    \item $\cC$ has a zero object, and for any zero object $0\in\cC$ and any $X\in\cC$, the map $0\to X$ is a cofibration
    \item Cofibrations are stable under pushout, so given any cofibration $X\hookrightarrow Y$ and any morphism $X\to Y$, the map $Y\to Y\cup_XZ$ is a cofibration
    \item All isomorphisms are weak equivalences
    \item Weak equivalences are closed under composition and hence form a subcategory
    \item Given a commutative diagram of the form
    \[
    \begin{tikzcd}
     Z \arrow[d, "\sim"] & X \arrow[d, "\sim"] \arrow[l] \arrow[r, hook] & Y \arrow[d, "\sim"]\\
     Z' & X' \arrow[l] \arrow[r, hook] & Y'
    \end{tikzcd}
    \]
    where the vertical arrows are weak equivalences and the horizontal arrows of the right square are cofibrations, one has that the induced map \[Y\cup_XZ\overset{\sim}{\to} Y'\cup_{X'}Z'\]
    is a weak equivalence as well.
\end{enumerate}
\end{defn}

There is also a natural notion of functor between Waldhausen categories.

\begin{defn}
A functor $F:\cC\to\cC'$ between two Waldhausen categories is \emph{exact} if it preserves cofibrations, weak equivalences, and finite (homotopy) colimits.  The category of Waldhausen categories and exact functors will be denoted $\text{Wald}$.
\end{defn}

While Waldhausen categories provided one of the earliest settings for the algebraic K-theory of categories (preceded only by exact categories), in recent years, several other (often related) frameworks have been used.  One of the most important (for this and many other reasons) is that of stable quasicategories.

\begin{defn}
Suppose that one has a pointed quasicategory $\cC$.  Given any morphism $X\to Y$ in $\cC$, its \emph{kernel} and \emph{cokernel} are defined, if they exist, by the homotopy cartesian and cocartesian squares
\[
\begin{tikzcd}
 \Ker{f} \arrow[d] \arrow[r] & X\arrow[d, "f"]\\
 0 \arrow[r] & Y
\end{tikzcd}
\text{ and }
\begin{tikzcd}
 X \arrow[d] \arrow[r, "f"] & Y\arrow[d]\\
 0 \arrow[r] & \Coker{f}
\end{tikzcd}
\]
respectively.

In general, an arbitrary commutative square 
\[
\begin{tikzcd}
 X \arrow[d] \arrow[r, "f"] & Y\arrow[d, "g"]\\
 0 \arrow[r] & Z
\end{tikzcd}
\]
is called a \emph{triangle}.  If it is cartesian, it is called an \emph{exact triangle} and if it is cocartesian, it is called a \emph{coexact triangle}.
\end{defn}
\begin{defn}
A quasicategory $\cC$ is \emph{stable} if it satisfies the following conditions:
\begin{itemize}
    \item $\cC$ is pointed (i.e. has a zero object)
    \item Every morphism in $\cC$ has a kernel and a cokernel
    \item Every exact triangle is coexact and every coexact triangle is exact.
\end{itemize}
An \emph{exact functor} $F:\cA\to\cB$ between stable quasicategories is one which preserves finite colimits.
\end{defn}
This extremely simple definition belies its incredible depth.  In particular, the homotopy category of a stable quasicategory is a triangulated category, and essentially every important example of triangulated categories arises as such a homotopy category.  We will not describe the theory of stable quasicategories here (a comprehensive guide is that of \cite{Lur2}).

Before continuing, we note the following proposition/definition

\begin{prop}
There is a Quillen equivalence between the model category of simplicial sets equipped with the Joyal model structure and the model category of simplicial categories equipped with the Bergner model structure:
\[
(\mathfrak C\dashv N^{hc}):\textbf{sSet}_{\text{Joyal}}\rightleftarrows\textbf{sSetCat}_{\text{Bergner}}.
\]
The right adjoint is known as the \emph{homotopy coherent nerve} and takes Kan-enriched categories to quasicategories.  Furthermore, it fits into a strictly commutative triangle
\[
\begin{tikzcd}
 \Cat\arrow[rr, "N"]\arrow[dr, hook] & & \textbf{sSet}\\
  & \textbf{sSetCat}\arrow[ur, "N^{hc}"]
 \end{tikzcd}.
\]
In other words, the classical nerve of a category may be factored into the inclusion of categories into simplicial categories as simplicial categories with discrete $\Hom$ simplicial sets followed by the homotopy coherent nerve.
\end{prop}

\begin{proof}
See \cite{Ber} section 7.4.
\end{proof}

This allows us a direct comparison between our chosen model of $\infty$-categories and the one which is most often used.

\begin{defn}
A stable $\infty$-category (incarnated as a fibrant $S$-category) will be one which maps under the homotopy coherent nerve to a stable quasicategory.  It should be noted that these will be precisely the $S$-categories for which the above three axioms hold, but the word homotopy is inserted in front of the word (co)cartesian in the definition of (co)exact triangles.

An \emph{exact functor} $F:\cA\to\cB$ between two stable $\infty$-categories will be one which preserves finite homotopy colimits.
\end{defn}

The last class of categories we will look at are so-called SW-categories and their precursors.  These categories were introduced as a way of categorifying the notion of subtraction present in settings such as decomposing varieties into open and closed complementary subsets and decomposition of polytopes.  In particular, they allow us a method of lifting universal Euler characteristics to the level of spectra, and will be extremely important later in the paper.  To define SW-categories, we must first define a few prerequisite categorical structures along the way.  It should also be noted that there are other categorical frameworks that deal with the concept of subtraction and scissors congruence \textemdash most notably the notion of an assembler category first introduced by Inna Zakharevich in \cite{Zakh} .  Assembler categories are more natural to define, but are slightly less easy to map out of (into targets such as Waldhausen categories or stable $\infty$-categories, for example).  For this latter reason, our preferred setting is that of SW-categories.  It should also be noted that for our purposes, subtractive categories are really enough, as our "weak equivalences" will merely be isomorphisms.

\begin{defn}
A \emph{pre-subtractive category} $\cC$ is a category equipped with two wide subcategories $\cof(\cC)$ and $\comp(\cC)$ referred to as \emph{cofibrations} and \emph{complements} (morphisms in $\cof(\cC)$ are denoted by $\hookrightarrow$ and morphisms in $\comp(\cC)$ are denoted by $\overset{\circ}{\to}$) and equipped with a subclass $\sub(\cC)$ of diagrams of the form $Z\hookrightarrow X\overset{\circ}{\leftarrow}Y$ referred to as \emph{subtraction sequences}.  These are all required to satisfy the following axioms:
\begin{itemize}
    \item $\cC$ has an initial object (often referred to via $\emptyset$ in practice)
    \item $\cC$ has finite coproducts, and for every $X, Y\in\cC$, one has that $X\to X\coprod Y$ is both a cofibration and a complement
    \item Pullbacks along cofibrations and complements exist and are cofibrations and complements respectively
    \item For every $X, Y\in\cC$, one has $X\hookrightarrow X\coprod Y\overset{\circ}{\leftarrow}Y\in\sub(\cC)$
    \item Every cofibration $X\hookrightarrow Y$ participates in a subtraction sequence $Z\hookrightarrow X\overset{\circ}{\leftarrow}Y$ which is unique up to unique isomorphism.  We denote this unique $Y$ by $X-Z$.  The same statement is true in reverse for every complement
    \item Any cartesian square of cofibrations
    \[
    \begin{tikzcd}
     Z \arrow[d, hook] \arrow[r, hook] & X \arrow[d, hook]\\
     Y \arrow[r, hook] & W
    \end{tikzcd}
    \]
    can be completed into a diagram of the form
    \[
    \begin{tikzcd}
      Z \arrow[d, hook] \arrow[r, hook] & X \arrow[d, hook] & X-Z \arrow[l, "\circ"] \arrow[d, hook]\\
     Y \arrow[r, hook] & W & W-Y \arrow[l, "\circ"]\\
     Y-Z \arrow[u, "\circ"] \arrow[r, hook] & W-X \arrow[u, "\circ"] & S \arrow[l, "\circ"] \arrow[u, "\circ"]
    \end{tikzcd}
    \]
    where $S:=(W-X)-(Y-Z)\cong(W-Y)-(X-Z)$ in such a way that every row and column will be a subtraction sequence, the bottom row and rightmost column are uniquely determined once choices of the complement are made, and the bottom-right square will also be cartesian.  The dual statement holds for cartesian squares of complements
    \item Subtraction is stable under pullback, or in other words, given any subtraction sequence $Z\hookrightarrow X\overset{\circ}{\leftarrow}Y$ and any morphism $W\to X$ in $\cC$, one has that $Z\times_XW\hookrightarrow W\overset{\circ}{\leftarrow}Y\times_XW$ is also a subtraction sequence
\end{itemize}
\end{defn}

We will often merely refer to a pre-subtractive category as $\cC$ instead of specifying all of the attendant data.  The most important examples of pre-subtractive categories for us are actually subtractive categories, as discussed immediately below.

\begin{defn}
A \emph{subtractive category} is a pre-subtractive category $\cC$ which additionally satisfies the following axioms:
\begin{itemize}
    \item The pushout of a diagram in which both legs are cofibrations exists and satisfies base-change (the created maps are also cofibrations).  Furthermore, cocartesian diagrams of this form are cartesian as well
    \item Given a cartesian diagram of cofibrations
    \[
    \begin{tikzcd}
     Z \arrow[d, hook] \arrow[r, hook] & X \arrow[d, hook]\\
     Y \arrow[r, hook] & W
    \end{tikzcd},
    \]
    the natural map $X\coprod_ZY\hookrightarrow W$ must also be a cofibration
    \item Given a diagram of the form
     \[
    \begin{tikzcd}
      X' \arrow[d, hook] & W' \arrow[d, hook] \arrow[l, hook] \arrow[r, hook] & Y' \arrow[d, hook]\\
     X & W \arrow[l, hook] \arrow[r, hook] & Y \\
     X'' \arrow[u, "\circ"] & W'' \arrow[u, "\circ"] \arrow[l, hook] \arrow[r, hook] & Y'' \arrow[u, "\circ"]
    \end{tikzcd}
    \]
    in which all the columns are subtraction sequences and all of the horizontal morphisms are cofibrations, the sequence $X'\coprod_{W'}Y'\hookrightarrow W\coprod_WY\overset{\circ}{\leftarrow}X''\coprod_{W''}Y''$ is a subtraction sequence
\end{itemize}
\end{defn}

In particular, it is proven in \cite{Cam} proposition 3.28 that given a base scheme $S$, the categories $\Var_S$ and $\Sch_S$ of $S$-varieties and $S$-schemes, respectively, have the structure of subtractive categories with cofibrations being closed immersions, complements being open immersions, and subtractions given by decomposition of varieties into complementary closed and open subvarieties.

\begin{defn}
An \emph{exact functor} $F:\cC\to\cC'$ between subtractive categories is a functor which satisfies the following properties:
\begin{itemize}
    \item $F$ preserves zero objects
    \item $F$ preserves subtraction sequences
    \item $F$ preserves cocartesian squares
\end{itemize}
The category of subtractive categories and exact functors is denoted $\text{SubCat}$.
\end{defn}

As one can imagine, the inclusion functor $\Var_S\hookrightarrow\Sch_S$ is in fact an exact functor of subtractve categories.

Now, we have all that is needed for the current paper, but if we happen to want a notion of weak equivalence that "plays nicely" with our other categories, we may enlarge our definition somewhat.

\begin{defn}
An \emph{SW-category} is a subtractive category $\cC$ equipped with an additional distinguished wide subcategory $w\cC$ of \emph{weak equivalences} (arrows in which are denoted $\overset{\sim}{\to}$) subject to the following conditions:
\begin{itemize}
    \item $w\cC$ contains all isomorphisms
    \item Given a commutative diagram of the form
     \[
    \begin{tikzcd}
      X' \arrow[d, "\sim"] & W' \arrow[d, "\sim"] \arrow[l, hook] \arrow[r, hook] & Y' \arrow[d, "\sim"]\\
     X & W \arrow[l, hook] \arrow[r, hook] & Y
    \end{tikzcd}
    \]
    in which horizontal arrows are cofibrations and vertical arrows are weak equivalences, one has that the resulting map $X'\coprod_{W'}Y'\overset{\sim}{\to}X\coprod_WY$ is a weak equivalence
    \item Weak equivalences respect subtraction, or in other words, given any commutative square
    \[
    \begin{tikzcd}
     Z \arrow[r, hook] \arrow[d, "\sim"] & X \arrow[d, "\sim"]\\
     W \arrow[r, hook] & Y
    \end{tikzcd}
    \]
    we may complete it into a commutative diagram 
    \[
    \begin{tikzcd}
     Z \arrow[r, hook] \arrow[d, "\sim"] & X \arrow[d, "\sim"] & X-Z \arrow[l, "\circ"] \arrow[d, "\sim"]\\
     W \arrow[r, hook] & Y & Y-W \arrow[l, "\circ"]
    \end{tikzcd}.
    \]
\end{itemize}
\end{defn}

\begin{defn}
A functor $F:\cC\to\cC'$ of SW-categories is \emph{exact} if it is exact as a functor of subtractive categories and preserves weak equivalences.  The category of SW-categories and weak equivalences will be denoted $\text{SW-Cat}$.
\end{defn}

If it seems thus far that we have been specifying a good deal more data than in the Waldhausen case, you are correct.  In truth, cutting and pasting/subtraction data does not play particularly nicely with homotopy, and several strong axioms are needed to ensure that we can say anything at all.  Luckily, in spite of the strength of the axioms introduced, the most important examples, namely $\Var_S$ and $\Sch_S$, satisfy these strong axioms, allowing us to define K-theory, among other things.

\subsection{Maps from Waldhausen and SW $1$-categories into stable $\infty$-categories}
In the current subsection, we will only develop as much of the comparison theory as is needed for this paper.  All of this can be developed more generally, as has already been done in the case of Waldhausen $\infty$-categories by Clark Barwick.  That said, it ought to be possible to replace the $1$-categorical theory of SW-categories with an $\infty$-categorical analogue, although we do not do this here.  Indeed, this generality would be unnecessary for our current setting, as the sub-$\infty$-category of derived prestacks generated by algebraic varieties is discrete (or in other words, forms a $1$-category).  All our $1$-categories will be implicitly included into $\infty$-categories (in other words, we are making use of the inclusion of $1$-categories into simplicial categories as the simplicial categories with discrete $\Hom$ simplicial sets).  Furthermore, for this subsection in particular, we will make the assumption that all of our $\infty$-categories are small as a precaution to prevent swindles and ensure everything is well defined.

Before we begin discussing maps between the different types of $\infty$-categories we will be using, let us say a little bit about how to extract K-theory from each one.

\begin{defn}
Define $\text{Ar}[n]$ to be the full subcategory of $[n]\times[n]$ consisting of $(i, j)\in[n]\times[n]$ with $i\le j$ and $\tilde{\text{Ar}}[n]$ to be the full subcategory of $[n]^{op}\times[n]$ consisting of $(i, j)\in[n]^{op}\times[n]$ with $i\le j$.  The former is used in the construction of Waldhausen's K-theory and the K-theory of pointed $\infty$-categories (specifically, stable $\infty$-categories), while the latter is used in the construction of K-theory of SW-categories.
\end{defn}

Waldhausen categories were of course first defined as a general setting for algebraic K-theory.  Waldhausen's $S_\bullet$-construction is quite well-known at this point, so we will omit the basic variation here (see, for example, \cite{Wei} or \cite{Wald}).  That said, we will make relatively heavy use of a variation known as the $S'_\bullet$-construction which is defined for Waldhausen categories which admit a functorial factorization of any morphism into a cofibration followed by a weak equivalence.  We recall briefly that, in such a Waldhausen category, a homotopy cocartesian square is weakly equivalent via a zigzag of equivalences to a pushout square with one leg a cofibration and a weak cofibration is a map equivalent to a cofibration by a zigzag of weak equivalences.

\begin{defn}
We begin by defining $S'_n\cC$ to be the full subcategory of $\text{Fun}(\text{Ar}[n], \cC)$ spanned by functors $F:\text{Ar}[n]\to\cC$ such that 
\begin{itemize}
    \item $F(i, i)$ is a zero object for all $i$ between $0$ and $n$
    \item $F(i, j)\to F(i, k)$ is a weak cofibration for any $i\le j\le k$
    \item Whenever $i<j<k$, one has that the diagram
    \[
    \begin{tikzcd}
     F(i, j)\arrow[d]\arrow[r] & F(i, k)\arrow[d]\\
     F(j, j)\arrow[r] & F(j, k)
    \end{tikzcd}
    \]
    is homotopy cocartesian
\end{itemize}
These categories bundle together to form a simplicial object in categories $S'_\bullet\cC$.  This construction is functorial in exact functors.
\end{defn}

Each $S'_n\cC$ may naturally be given the structure of a Waldhausen category itself, so we may define a multisimplicial category $(S'_\bullet)^n\cC$ by iterating the construction at every simplicial level $n$ times.  We may then define the K-theory spectrum $K(\cC)$ of this Waldhausen category with $n$-th space $|w(S'_\bullet)^n\cC|$.  We do not use a different notation for this K-theory space, as it coincides with the standard Waldhausen K-theory in the case we described before.

We will also define K-theory for general pointed finitely cocomplete quasicategories. Given any pointed finitely cocomplete quasicategory $\cA$, one can define K-theory in the following way.

\begin{defn}
We begin by defining $\text{Gap}([n],\cA)$ to be the full subcategory of $\text{Fun}(\text{Ar}[n], \cA)$ spanned by functors $F:\text{Ar}[n]\to\cA$ such that 
\begin{itemize}
    \item $F(i, i)$ is a zero object for all $i$ between $0$ and $n$
    \item Whenever $i<j<k$, one has that the diagram
    \[
    \begin{tikzcd}
     F(i, j)\arrow[d]\arrow[r] & F(i, k)\arrow[d]\\
     F(j, j)\arrow[r] & F(j, k)
    \end{tikzcd}
    \]
    is cocartesian
\end{itemize}
\end{defn}

Note that $\text{Gap}([n], \cA)$ is always pointed and finitely cocomplete via \cite{BGT}, and when $\cA$ is stable, it is as well.

\begin{defn}
Given $\cA$ pointed and finitely cocomplete (resp. stable), we define the simplicial pointed finitely cocomplete (resp. stable) quasicategory $S^\infty_\bullet\cA$ via  $S^\infty_n\cA:=\text{Gap}([n], \cA)$ on objects, with face maps given by the composition of the arrows into and out of the objects on the $i$th row and column and degeneracy maps given by inserting a copy of the identity in position $i$ on rows and columns.  In particular, denoting by $\iota$ the largest internal $\infty$-groupoid functor, this also yields a simplicial $\infty$-groupoid $\iota S^\infty_\bullet\cA$.
\end{defn}

Note that we can actually iterate the construction above to give a multisimplicial pointed finitely cocomplete (resp. stable) $\infty$-category $(S^\infty_\bullet)^n\cA$ and corresponding multisimplicial $\infty$-groupoid $\iota(S^\infty_\bullet)^n\cA$.

\begin{defn}
Given a pointed finitely cocomplete (resp. stable) $\infty$-category $\cA$, we define its \emph{K-theory spectrum} $K(\cA)$ as (the fibrant-cofibrant replacement of what) follows.  For every $n$, the $n$th space in the spectrum is $|\iota(S^\infty_\bullet)^n\cA|$. Noting that $\text{Gap}([0], \cA)$ is contractible and that $\text{Gap}([1], \cA)$ is naturally equivalent to $\cA$, we obtain a natural map $S^1\wedge|\iota\cA|\to|S^\infty_\bullet\cA|$, which induces all of our connecting maps.

This construction is natural in exact functors. 
\end{defn}

When restricted to small stable $\infty$-categories equipped with the Lurie tensor product, the above construction defines a symmetric monoidal functor to the category of spectra.  In particular, this implies given a symmetric monoidal small $\infty$-category $\cA$, that the K-theory spectrum $K(\cA)$ is a $\bE_\infty$-ring spectrum.

We mostly introduce the notion of K-theory of quasicategories as a "standard" K-theory of $\infty$-categories to compare against, simply because this K-theory has had many nice properties proved about it already.  Before this, we introduce the K-theory of SW-categories and describe how to compare it to the K-theory of $\infty$-categories (incarnated as $S$-categories), which we also introduce.

\begin{defn}
Given any pointed finitely cocomplete $\infty$-category $\cA$ (considered as a Kan-enriched category), one may consider its category of pointed simplicial presheaves $\cP(\cA)$ equipped with the projective model structure.  From here, one can left Bousfeld localize this category to form the simplicial model category $\cP_{ex}(\cA)$ whose local objects are those functors $\cA^{op}$ to simplicial sets which commute with finite (homotopy) colimits.  Finally, we can construct a Waldhausen subcategory $\cM(\cA)$ of $\cP_{ex}(\cA)$ which consists of those objects which are cofibrant and weakly equivalent to representable presheaves.  From here, we may define \emph{the K-theory spectrum of $\cA$} to be the the Waldhausen K-theory $K(\cA):=K(\cM(\cA))$.
\end{defn}

It is this definition of K-theory of $\infty$-categories that we will compare with the K-theory of SW-categories.  Note that if one has a morphism $\cA\to\cB$ of pointed, finitely homotopy cocomplete $\infty$-categories which commutes with the relevant structures (a so-called weakly exact functor), one obtains an exact morphism $\cM(\cA)\to\cM(\cB)$ of Waldhausen categories by restricting the left Kan extension of our original morphism.  In this way, we obtain a functor from pointed finitely homotopy cocomplete $\infty$-categories and weakly exact functors to (strongly saturated) Waldhausen categories and exact functors.  This, in turn, may be composed with K-theory, which shows that the above definition of K-theory is functorial (\cite{Cis} page 544). 

This will allow us an easier way of defining SW-exact functors to $\infty$-categories without actually needing to enlarge our definition all that much.

Before we continue, let us begin by comparing the above two constructions on nice $\infty$-categories to ensure they coincide.  

\begin{thm}
Given a simplicial model category $\cC$ and a small full subcategory $\cA\subseteq\cC$ whose objects are all cofibrant, which admits all homotopy colimits, and whose underlying category inherits the structure of a Waldhausen category from the model structure on $\cC$, one has equivalences
\[
K(\cA)\simeq K(\cM(\cA))\simeq K(N^{hc}(\cA^{cf})),
\]
where $\cA^{cf}$ is the full subcategory of $\cA$ generated by objects which are both cofibrant and fibrant in $\cC$.
\end{thm}
\begin{proof}
This combines \cite{BGT} Theorems 7.8 and 7.11 and Corollary 7.12.
\end{proof}

In particular, this will be true for all of our categories of interest, as any stable $S$-category embeds into a stable simplicial model category (see \cite{ToeVez} page 789 paragraph 1).

\begin{defn}
Suppose that $\cC$ is an SW-category.  In analogy with the Waldhausen case, we may define a simplicial object in categories $\tilde S_\bullet\cC$ as follows.  We start by defining it as a simplicial set.  For each $n$, define $\tilde S_n\cC$ to the full subcategory of $\text{Fun}(\tilde{\text{Ar}}[n], \cC)$ on the functors $F:\tilde{\text{Ar}}[n]\to\cC$ satisfying the following conditions:
\begin{itemize}
    \item $F(i, i)$ is initial for any $i$ between $0$ and $n$
    \item Whenever $j<k$, one has $F(i, j)\to F(i, k)$ is in $\cof(\cC)$
    \item Whenever $i<j<k$ one has that 
    \[
    F(i, j)\rightarrow F(i, k)\leftarrow F(j, k)
    \]
    is a subtraction sequence.
\end{itemize}
The face maps $d_k:\tilde S_n\cC\to\tilde S_{n-1}\cC$ are given by deleting the $k$th row and column of the requisite functors from $\tilde{\text{Ar}}[n]$, and composing the appropriate morphisms to yield a functor from $\tilde{\text{Ar}}[n-1]$ into $\cC$.  The degeneracy maps are given by inserting identity maps into the $i$th row and column.  Furthermore, each $\tilde S_n\cC$ naturally has the structure of an SW-category (we can check whether a morphism of diagrams is a cofibration, complement, or weak equivalence object-wise), so we can iterate this construction to an $n$-ary simplicial structure $\tilde S^n_\bullet\cC$.
\end{defn}

\begin{defn}
Given an SW-category $\cC$, one can define its \emph{K-theory spectrum} $K(\cC)$ as (the fibrant-cofibrant repacement of what) follows.  Define the $n$th space of the spectrum to be $|w\tilde S_\bullet^n\cC|$.  By \cite{Cam} corollary 4.17, one gets a natural map $|w\cC|\to\Omega|w\tilde S_\bullet\cC|$ which, while not an equivalence, becomes one on higher levels, so we actually have a spectrum.
\end{defn}

As an example, $|w\tilde S_\bullet\cC|\overset{\sim}{\to}\Omega|w\tilde S_\bullet\tilde S_\bullet\cC|$.  The spectrum $K(\cC)$ also has the natural structure of a $\bE_\infty$-ring spectrum whenever $\cC$ is equipped with a symmetric monoidal structure compatible with subtraction (\cite{Cam} theorem 5.14).  This construction defines a functor from SW-categories equipped with exact functors of SW-categories to spectra.

\begin{defn}
Let $\cC$ be an SW-category and $\cA$ be a pointed homotopy cocomplete $\infty$-category.  A \emph{weakly W-exact functor} $F:\cC\to\cA$ actually consists of a triple $(F_!, F^!, F_w)$ of functors such that 
\begin{itemize}
    \item $F_!$ is a functor $F_!:\cof(\cC)\to\cA$.  We abbreviate $F_!(i)$ to $i_!$ for all cofibrations $i$
    \item $F^!$ is a functor $F^!:\comp(\cC)^{op}\to\cA$.  We abbreviate $F^!(j)$ to $j^!$ for all complement maps $j$
    \item $F_w$ is a functor $F_w:w\cC\to\iota(\cA)$.  We abbreviate $F_w(f)$ to $f_w$ for all weak equivalences
    \item For all objects $X\in\cC$, one has $F_!(X)=F^!(X)=F_w(X)=:F(X)$
    \item Given a cartesian square in $\cC$
    \[
    \begin{tikzcd}
 X\arrow[d, "i"]\arrow[r, hook, "j"] & Z \arrow[d, "i'"]\\
  Y \arrow[r, hook, "j'"]& W
\end{tikzcd}
    \]
    with horizontal morphisms cofibrations and vertical morphisms complement maps, one gets a commutative diagram
    \[
    \begin{tikzcd}
 F(X)\arrow[r, "j_!"] & F(Z) \\
  F(Y)\arrow[u, "i^!"]\arrow[r, "j'_!"]& F(W)\arrow[u, "i'^!"]
\end{tikzcd}
    \]
    \item For all subtraction sequences $Z\overset{i}{\hookrightarrow}X\overset{j}{\underset{\circ}{\leftarrow}}U$, one gets a homotopy cocartesian square
    \[
    \begin{tikzcd}
 F(Z)\arrow[d]\arrow[r, "i_!"] & F(X) \arrow[d, "j^!"]\\
   0 \arrow[r]& F(U)
\end{tikzcd}
    \]
    in $\cA$
    \item For all commutative squares 
    \[
    \begin{tikzcd}
 X\arrow[d, "g"]\arrow[r, hook, "f"] & Z \arrow[d, "g'"]\\
  Y \arrow[r, hook, "f'"]& W
\end{tikzcd}
    \]
    with vertical morphisms weak equivalences and horizontal morphisms cofibrations, one gets a commutative square
    \[
    \begin{tikzcd}
 F(X)\arrow[d, "g_w"]\arrow[r, "f_!"] & F(Z) \arrow[d, "g'_w"]\\
  F(Y) \arrow[r, "f'_!"]& F(W)
\end{tikzcd}
    \]
    in $\cA$. One gets an analogous diagram if one replaces cofibrations with complements
\end{itemize}
\end{defn}

Our reason for introducing such a map is that such a map induces a contractible space of maps of K-theory spectra.

\begin{rem}
Note that this is precisely the same as the definition of a weakly W-exact functor where the target is a Waldhausen category (\cite{CWZ} definition 2.17). If $F:\cC\to\cA$ is a weakly W-exact functor with target a Waldhausen category, it induces a map $K(\cF): K(\cC)\to K(\cA)$ (\cite{CWZ} proposition 2.19).  We make use of this in the following theorem.
\end{rem}

\begin{thm}
Given a weakly W-exact functor $F:\cC\to\cA$, it may be composed with the Yoneda embedding $\cA\to\cM(\cA)$ to obtain a weakly W-exact functor $\iota\circ F:\cC\to\cM(\cA)$ to the good Waldhausen $\cM(\cA)$ which yields a map on K-theory $K(\cC)\to K(\cA)$.  This construction is functorial in exact functors both of SW-categories and of pointed finitely homotopy cocomplete $\infty$-categories.
\end{thm}

\begin{proof}
Let us begin by noting that $\cM(\cA)$ has a natural enrichment as a simplicial category which makes the restriction of the Yoneda embedding $\iota:\cA\to\cM(\cA)$ into a functor of simplicial categories (and indeed one which preserves finite homotopy colimits by construction).  As such, the corresponding map $\iota\circ F$ remains weakly W-exact.  Furthermore, by construction, it must descend to a weakly W-exact map on the level of the underlying Waldhausen category $\cM(\cA)$.  Thus, one has a natural map on K-theory $K(\cC)\to K(\cA)$ induced from $\iota\circ F$.  Furthermore, one has that functoriality in weakly W-exact functors and weakly exact functors arises from the functoriality of the corresponding constructions on their respective K-theories.
\end{proof}

\section{The Six Functors Formalism and Derived Motivic Measures}
\subsection{A Brief Description of the Generalized Six Functors Formalism}

In this section, we will discuss some of the basic aspects of the $\infty$-categorical theory of six functors formalisms described by Adeel Khan in \cite{Khan} .  We will not go into detail regarding the proofs, and will only supply the details necessary for the sections that follow.  It should be noted that none of this section is novel.  Any interested reader is urged to peruse the stellar paper by Khan on the subject.  
Before we begin, we will detail a few conventions.

For the rest of the section $\cS$ will refer to a category of classical Noetherian schemes over some fixed Noetherian base. While Khan's original formalism works more generally for derived algebraic spaces, for the moment, we are only concerned with a more limited setup.  Note that these stronger conditions imply, among other things, that $\cS$ is actually a $1$-category (a discrete $\infty$-category).  Furthermore, we will assume that $\cS$ is such that for all $S\in\cS$, one has
\begin{itemize}
    \item $U\in\cS$ for every quasicompact open subspace $U\subseteq S$
    \item $Z\in\cS$ for every closed subspace $Z\subseteq S$
    \item $\bP(\cE)\in\cS$ for every finite locally free sheaf $\cE$ on $S$
\end{itemize}
We further fix a class of so-called \emph{admissible} morphisms in $\cS$, which we mandate to contain all open immersions and all projections $X\times \bP^n\to X$, to be closed under composition and base change, and to satisfy 2-out-of-3.  We denote by $\cA\subseteq\cS$ the subcategory of admissible morphisms.

If we consider a presheaf of $\infty$-categories $\bD^*$ on $\cS$, we will use the notation $\bD(S):=\bD^*(S)$ for every $S\in\cS$.  Furthermore, for every morphism $f:X\to Y\in\cS$, we will denote by
\[
f^*:\bD(Y)\to\bD(X)
\]
the inverse image $\bD^*(f)$ of $f$.

If, furthermore, $\bD^*$ takes values in presentable $\infty$-categories and colimit-preserving functors, we refer to $\bD$ as a \emph{presheaf of presentable $\infty$-categories}.  Note that this, in particular, means that $f^*$ must admit a right adjoint, which we denote by $f_*$ and call the direct image of $f$.

Finally, if $\bD^*$ further factors through symmetric monoidal presentable $\infty$-categories (presentable symmetric monoidal $\infty$-categories for which the tensor product commutes with colimits in each variable), we refer to $\bD^*$ as a \emph{presheaf of symmetric monoidal presentable $\infty$-categories}.  Given any $S\in\cS$, we will let $\otimes:\bD(S)\times\bD(S)\to\bD(S)$ denote the monoidal product and $\mathbf{1}_S$ denote the monoidal unit. Now, since $\otimes$ commutes with colimits in each variable, it admits a right-adjoint internal hom bifunctor $\underline{\Hom}:\bD(S)^{op}\times\bD(S)\to\bD(S)$.  From now on, when we talk about this last scenario, we will simply omit the upper $*$ from the notation unless we want to make it clear that we are referring to $\bD$ as presheaf with respect to $*$ morphisms.

\begin{defn}
A \emph{premotivic $\infty$-category} or \emph{$(*, \#, \otimes)$-formalism} on $(\cS, \cA)$ is a presheaf of symmetric monoidal presentable $\infty$-categories $\bD$ on $\cS$ which satisfies the following additional properties
\begin{itemize}
    \item For every morphism $f:T\to S$ in $\cA$, the inverse image functor admits a left-adjoint $f_\#:\bD(T)\to\bD(S)$ which we call the \emph{$\#$-direct image}
    \item $f_\#:\bD(T)\to\bD(S)$ is a morphism of $\bD(S)$-modules, where we note that $\bD(T)$ has a natural $\bD(S)$-module structure via the symmetric monoidal functor $f^*$.  In other words, $\bD$ satisfies the \emph{projection formula} with respect to $\#$-direct images
    \item $\bD$ satisfies \emph{admissible base change} for $\#$-direct image.  In other words, given any cartesian diagram of the form
 \[
    \begin{tikzcd}
 T'\arrow[d, "q"]\arrow[r, "g"]\arrow[dr, phantom, "\lrcorner", very near start] & S' \arrow[d, "p"]\\
  T \arrow[r, "f"]& S
\end{tikzcd}
    \]    
    with $p$ and $q$ admissible, then there is a natural equivalence of functors 
    \[
    \Ex_\#^*:q_\#g^*\overset{\sim}{\to} f^*p_\#
    \]
    \item Given any finite family $S_\alpha$ in $\cS$, the induced functor 
    \[
    \bD(\amalg_\alpha S_\alpha)\to\Pi_\alpha\bD(S_\alpha)
    \]
    from inclusion is an equivalence.  In other words, $\bD$ satisfies additivity
\end{itemize}
Generally speaking, when talking about premotivic $\infty$-categories on $\cS$, if admissible morphisms are not specified as part of the data, then we are making the assumption that the admissible morphisms are simply the smooth morphisms in $\cS$ (or $\cA=Sm$). 
\end{defn}

\begin{rem}
Note that a $(*, \#, \otimes)$-formalism is the same as a premotivic $\infty$-category in the sense of Cisinski-D\'eglise (hence the conflation of the two terms above), so we may directly import the notions of premotivic morphism and premotivic adjunction to this setting (indeed, we note that premotivic stable-symmetric monoidal model categories are a direct realization of our current situation in more classical language when specified to the stable case).  
\end{rem}

We will delay discussion of examples until the following section.

\begin{defn}
Let $\bD$ be a $(*, \#, \otimes)$-formalism on $(\cS, \cA)$, and take $S\in\cS$ and a finite locally free sheaf $\cE$ on $S$.  Denote the total space $\bV_S(\cE)$ by $E$, and let $p:E\to S$ be the projection and $s:S\to E$ be the zero section.  Supposing that $p$ is admissible, define the \emph{Thom twist} $\langle\cE\rangle$ to be the endofunctor on $\bD(S)$ given by
\[
\cF\mapsto \cF\langle\cE\rangle:=p_\# s_*(\cF).
\]
\end{defn}

\begin{defn}
A $(*, \#, \otimes)$-formalism on $(\cS, \cA)$ \emph{satisfies the Voevodsky conditions} if it satisfies the following conditions
\begin{itemize}
    \item \emph{Homotopy invariance:} for every $S\in \cS$ and every vector bundle $p:E\to S$ on S, the unit map
    \[
    \id\to p_*p^*
    \]
    is an equivalence
    \item \emph{Localization:} for every decomposition of a variety $X$ into a closed subvariety and a complementary open
    \[
    Z\overset{i}{\hookrightarrow} X\overset{j}{\hookleftarrow} U
    \]
    in $\cS$, $i_*$ is a fully faithful functor with essential image spanned by objects in the kernel of $j^*$
    \item \emph{Thom stability:} for every $S\in\cS$ and every locally free sheaf $\cE$ on $S$, the Thom twist endofunctor $\langle\cE\rangle$ is an equivalence on $\bD(S)$
\end{itemize}

We will also use the term \emph{motivic category over $(\cS, \cA)$} to describe a $(*, \#, \otimes)$-formalism over $(\cS, \cA)$ which satisfies the Voevodsky conditions, since the Voevodsky conditions are equivalent (in the triangulated case) to the conditions of Cisinski-D\'eglise under which a premotivic category defines a motivic category.
\end{defn}

\begin{rem}
If $\bD$ is a $(*, \#, \otimes)$-formalism on $\cS$ satisfying the Voevodsky conditions, then the $\infty$-categories $\bD$ are stable.
\end{rem}

Note that the features we have already discussed automatically imply some of the most characteristic features of the notion of a six functors formalism, namely, the exceptional operations.

\begin{thm}
Given any finite type morphism $f:\bD(X)\to\bD(Y)$, there exists an adjunction
\[
(f_!\dashv f^!):\bD(X)\rightleftarrows\bD(Y)
\]
and a natural transformation $\alpha_f:f_!\to f_*$ satisfying the following conditions:
\begin{itemize}
    \item There are canonical equivalences $f_!\simeq f_\#$ and $f^!\simeq f^*$ if $f$ is an open immersion
    \item $\alpha_f$ is an equivalence if $f$ is a proper morphism
    \item The functor $f_!$ satisfies base change, or in other words, given a cartesian diagram
    \[
    \begin{tikzcd}
 X'\arrow[d, "q"]\arrow[r, "g"]\arrow[dr, phantom, "\lrcorner", very near start] & Y' \arrow[d, "p"]\\
  X \arrow[r, "f"]& Y
\end{tikzcd}
    \]    
    in $\cS$, the natural transformations 
    \[
    \Ex_!^*:p^*f_!\to g_!q^*\text{ and }\Ex_*^!: q_*g^!\to f^!p_*
    \]
    are equivalences
    \item The functor $f_!$ satisfies the projection formula.  In other words, $f_!$ is a morphism of $\bD(Y)$-modules, where $\bD(X)$ is regarded as a $\bD(Y)$ module via the symmetric monoidal functor $f^*$.  Furthermore, the canonical morphisms
    \begin{align*}
        \cF\otimes f_!(\cG)&\to f_!(f^*(\cF)\otimes\cG),\\
        \underline{\Hom}(f^*(\cF), f^!(\cF'))&\to f^!(\underline{\Hom}(\cF, \cF')),\\
        f_*(\underline{\Hom}(\cF, f^!(\cG)))&\to\underline{\Hom}(f_!(\cF), \cG)\\
    \end{align*}
    are equivalences for all $\cF,\cF'\in\bD(X)$ and $\cG\in\bD(Y)$
\end{itemize}
\end{thm}

\begin{proof}
This is \cite{Khan} theorem 2.34.
\end{proof}

In fact, if $f$ is \'etale, then the natural map $f^!\to f^*$ is an equivalence.

Before we continue, at this point we should introduce a bit of notation.  Suppose that $f:X\to Y$ is a morphism of finite type in $\cS$.  By the above, we have adjunctions $f^*\dashv f_*$ and $f_!\dashv f^!$.  We will use the notation
\begin{itemize}
    \item $\eta_f^*:\id\to f_*f^*$ and $\epsilon_f^*:f^*f_*\to\id$ for the unit and counit of the first adjunction
    \item $\eta_f^!:\id\to f^!f_!$ and $\epsilon_f^!:f_!f^!\to\id$ for the unit and counit of the second adjunction
\end{itemize}
If, in addition, $f$ happens to be smooth, we have a third adjunction $f_\#\dashv f^*$, and we will denote its unit and counit by 
\[
\eta_f^\#:\id\to f^*f_\#\text{ and }\epsilon_f^\#:f_\#f^*\to\id
\]
if needed.
\begin{defn}
A premotivic category $\bD$ on $(\cS, \cA)$ is \emph{compactly generated} if
\begin{itemize}
    \item For every $S\in\cS$, the $\infty$-category $\bD(S)$ is compactly generated
    \item For every morphism $f:T\to S$ in $\cS$, the inverse image functor $f^*:\bD(S)\to\bD(T)$ is a compact functor (preserves compact objects)
\end{itemize}
\end{defn}

\begin{defn}
Given a premotivic category $\bD$ over $(\cS, \cA)$, we refer to $\bD$ as \emph{continuous} if for every cofiltered system of affine schemes $(S_\alpha)_\alpha$ in $\cS$ with limit $S$, the canonical functor
\[
\varprojlim_\alpha\bD(S_\alpha)\to\bD(S)
\]
is an equivalence.
\end{defn}

In practice, essentially all of the (pre)motivic categories we encounter will be compactly generated, and many will be continuous as well.  There will be more on compact (and in particular constructible) generation towards the end of the section.

\begin{thm}
Suppose that $\bD$ is a motivic category over $(\cS, \cA)$ and that one has a cartesian diagram
 \[
    \begin{tikzcd}
 X'\arrow[d, "q"]\arrow[r, "g"]\arrow[dr, phantom, "\lrcorner", very near start] & Y' \arrow[d, "p"]\\
  X \arrow[r, "f"]& Y
\end{tikzcd}
    \]    
    in $\cS$.  Then we have the following different forms of base change in addition to those discussed before (presented in their most general forms):
    \begin{itemize}
        \item \emph{Proper base change:}  If $f$ is a proper morphism, then there is a canonical equivalence 
        \[
        \Ex_*^*:p^*f_*\to g_*q^*
        \]
        of functors $\bD(X)\to\bD(Y')$
        \item \emph{Smooth-proper base change:} If $f$ is a proper morphism and $p$ and $q$ are smooth, then there is a canonical equivalence
        \[
        \Ex_{\#*}:p_\#g_*\to f_*q_\#
        \]
        of functors $\bD(X')\to\bD(Y)$
        \item \emph{Finite type-smooth base change:} If $f$ is of finite type and $p$ and $q$ are smooth, then there is a natural equivalence
        \[
        \Ex^{*!}:q^*f^!\to g^!p^*
        \]
        of functors $\bD(Y)\to\bD(X')$
        \item \emph{Finite type-proper base change:} If $f$ is of finite type and $p$ is proper, then there is a natural equivalence
        \[
        \Ex_{!*}:f_!q_*\to p_*g_!
        \]
        of functors $\bD(X')\to\bD(Y)$
    \end{itemize}
\end{thm}

\begin{proof}
This consists of various statements in \cite{Khan} theorem 2.24, corollary 2.37, and corollary 2.39.
\end{proof}

Really the crux of many of the proofs in the following section will be the various forms of base change we have discussed.

\begin{rem}
In what follows, noting that Tate twists generally commute with all of the six operations, we abuse notation by doing things such as writing $f^*\langle\cE\rangle$ instead of $\langle f^*\cE\rangle\circ f^*$, etc.
\end{rem}

\begin{thm}
Consider $S\in\cS$ and two smooth $S$-schemes $p:X\to S$ and $q:Y\to S$.  Then one has the following two results:
\begin{itemize}
    \item \emph{Relative purity:} If $X$ and $Y$ are connected via a closed immersion $i:X\hookrightarrow Y$ over $S$, then there is a canonical isomorphism 
\[
q_\#i_*\simeq p_\#\langle\cN_{X/Y}\rangle,
\]
with $\cN_{X/Y}$ the conormal sheaf of $i$
    \item If $f:X\to Y$ is an unrammified morphism over $S$, then there is a canonical isomorphism 
    \[
    f^!q^*\simeq p^*\langle\cL_{X/Y}\rangle
    \]
    where $\cL_{X/Y}$ is the relative cotangent complex of $f$
\end{itemize}
\end{thm}

\begin{proof}
This is \cite{Khan} theorem 2.25 and 2.43.
\end{proof}

In particular, these can be used to conclude two important results.

\begin{thm}
\emph{Atiyah duality:} If $f:X\to Y$ in $\cS$ is smooth and proper, then one has a canonical morphism of functors 
\[
\epsilon_f:f_\#\langle\cL_f\rangle\to f_*,
\]
where $\cL_f$ is the cotangent complex of $f$.
\end{thm}

\begin{proof}
This is \cite{Khan} theorem 2.24 (iii).
\end{proof}
This theorem is very interesting, as it implies, among other things, that the left and right adjoints of $f^*$ are related by a Thom twist when $f$ is smooth and proper.

\begin{thm}
\emph{Purity:} If $\bD$ is a motivic category over $\cS$, then for any smooth morphism $f:X\to Y$, one has a canonical equivalence
\[
\text{pur}_f:f^!\to f^*\langle\cL_f\rangle
\]
of functors $\bD(Y)\to\bD(X)$, thus generalizing our previous result on \'etale morphisms from before.
\end{thm}

\begin{proof}
This is \cite{Khan} theorem 2.44.
\end{proof}

We will also make very heavy use of the following consequence of localization.

\begin{prop}
Let $\bD$ be a premotivic category on $\cS$, and suppose that one has a closed immersion $i$ with open complement $j$ of the form
\[
Z\overset{i}{\hookrightarrow}X\overset{j}{\hookleftarrow}U.
\]
Then the functors $j_*$ and $j_\#$ are fully faithful and furthermore one has $j^*i_*\simeq0$ and $i^*j_\#\simeq0$.  If, furthermore $\bD$ satisfies the localization property, then one has the canonical cofiber sequences
\[
j_\#j^*\to\id\to i_*i^*\text{ and }i_*i^!\to\id\to j_*j^*
\]
(there is a way to define $i^!$ in this case without assuming that $\bD$ is motivic, simply as the homotopy fiber of $i^*\to i^*j_*j^*$).  If, in fact, $\bD$ happens to be motivic, then the above cofiber sequences are equivalent to
\[
j_!j^!\to\id\to i_*i^*\text{ and }i_!i^!\to\id\to j_*j^*.
\]
\end{prop}
\begin{proof}
This is \cite{Khan} remark 2.9
\end{proof}

\bigskip

We now come to a very important definition.  The subcategory of constructible objects will be what allows us to extract meaningful K-theory from our six functors formalism and not fall prey to swindles. 

\begin{defn}
Consider a motivic category $\bD$ over $\cS$ and a scheme $S\in\cS$.  For any $\cF\in\bD(S)$, we say that $\cF$ is \emph{constructible} if it satisfies one of the following equivalent conditions:
\begin{itemize}
    \item $\cF$ lies in the thick subcategory generated by $f_\#f^*(\mathbf{1}_S)\langle-n\rangle$ (or $f_!f^!(\mathbf{1}_S)\langle-n\rangle$ via purity) with $f:X\to S$ a smooth morphism of finite presentation and $n\in\Z_{\ge0}$
    \item $\cF$ lies in the thick subcategory generated by $f_\#f^*(\mathbf{1}_S)\langle-n\rangle$ (or  $f_!f^!(\mathbf{1}_S)\langle-n\rangle$ via purity) with $f:X\to S$ a smooth morphism of finite presentation with $X$ affine and $n\in\Z_{\ge0}$
\end{itemize}
The \emph{subcategory of constructible objects over $S$} is denoted by $\bD_\cons(S)\subseteq\bD(S)$.
\end{defn}

Since it is this subcategory of constructible objects that we will need to extract meaningful K-theory in what follows, we will list here several of its important properties.

\begin{defn}
We say that a motivic category $\bD$ is constructibly generated if it is compactly generated and every constructible object is compact.  It then follows that the compactness is equivalent to constructibility (\cite{Khan} definition 2.5.5). 
\end{defn}

\begin{prop}
The property of constructibility in $\bD$ is stable under
\begin{itemize}
    \item Tensor product with any constructible object
    \item Inverse image along any morphism in $\cS$
    \item $\#$-direct image along any finitely presented smooth morphism in $\cS$
    \item Thom twist by a perfect complex
    \item Exceptional direct image along any finite type morphism in $\cS$
\end{itemize}
\end{prop}
\begin{proof}
The first three statements comprise \cite{Khan} proposition 2.57, the third comprises \cite{Khan} proposition 2.60, and the fourth comprises \cite{Khan} theorem 2.61.
\end{proof}

This implies that $\bD^*_\cons$ is a presheaf of symmetric monoidal essentially small $\infty$-categories on $\cS$.

\begin{cor}
If $i:Z\hookrightarrow X$ is a closed immersion, the property of constructibility in $\bD(X)$ is stable under the endofunctor $i_*i^*$.
\end{cor}
\begin{proof}
This is \cite{Khan} corollary 2.58.
\end{proof}

\subsection{Some Examples of Motivic $\infty$-Categories}

Perhaps the canonical example of a motivic $\infty$-category over $\cS$ is the stable motivic homotopy category $\SH$.  We will go through the definition here.  We will fix a subcategory $\cS$ of schemes over some base $B$ and an admissible subcategory $\cA$ of $\Sch_B$ for the rest of this section.  We require that $\cA$ satisfy the following properties:
\begin{itemize}
    \item $\cA$ is an essentially small full subcategory of $\Sch_B$
    \item For every $S\in\cS$, one has that the over category $\cA_{/S}$ contains $S$
    \item If $X\in\cA_{/S}$ and $Y$ is \'etale and of finite presentation over $S$, then $Y\in\cA_{/S}$
    \item If $X\in\cA_{/S}$, then $X\times\bA^1\in\cA_{/S}$
    \item If $X, Y\in\cA_{/S}$, then $X\times_SY\in\cA_{/S}$
\end{itemize}

Note that we do not necessarily require $\cA\subseteq\cS$ in this section (in fact, our typical choice for $\cA$ will merely be $\Sm_B$).  

We construct the stable motivic homotopy category in several steps.  The first order of business is to construct the \emph{unstable motivic homotopy category} $\textbf H(\cA_{/S})$ for any $S\in\cS$. 

\begin{defn}
Given any $S\in\cS$, an \emph{$\cA$-fibered space over $S$} is a presheaf on $\cA_{/S}$ with values in a suitable category of spaces $\textbf S$, such as Kan complexes.  We denote the $\infty$-category of such presheaves as $\PrSh_\textbf S(\cA_{/S})$.  
\begin{itemize}
    \item An element $\cF\in\PrSh_\textbf S(\cA_{/S})$ \emph{satisfies Nisnevich descent} if it satisfies \u Cech descent with respect to the restriction of the Nisnevich topology restricted to $\cA_{/S}$.  In other words, for all $X\in\cA_{/S}$ (suppressing structure morphism) and all Nisnevich coverings $\{U_\alpha\to X\}_\alpha$, the diagram  
    \[
    \cF(X)\to\prod_\alpha\cF(U_\alpha)\rightrightarrows\prod_{\beta,\gamma}\cF(U_\beta\times_XU_\gamma)\triplerightarrow\cdots
    \]
    exibits $\cF(X)$ as its homotopy limit
    \item An element $\cF\in\PrSh_\textbf S(\cA_{/S})$ \emph{is $\bA^1$-homotopy invariant} if for any $X\in\cA_{/S}$, the natural map
    \[
    \cF(X)\to\cF(X\times_S\bA^1)
    \]
    is an equivalence
\end{itemize}
An element $\cF\in\PrSh_\textbf S(\cA_{/S})$ is referred to as \emph{motivic} if it satisfies Nisnevich descent and $\bA^1$-homotopy invariance.  The category of motivic $\cA$-fibered spaces is referred to as the \emph{unstable motivic homotopy category over S with respect to $\cA$} and is denoted $\textbf{H}(\cA_S)$.  It may equally be seen as the successive left Bousfeld localization of $\PrSh_\textbf S(\cA_{/S})$ at the classes of morphisms $\cF(X)\to\text{Tot}(\{U_\alpha\}, \cF)$ for any $\cF\in\PrSh_\textbf S(\cA_{/S})$, any $X\in\cA_{/S}$, and any Nisnevich cover $\{U_\alpha\to X\}_\alpha$ (here $\text{Tot}(\{U_\alpha\}, \cF)$ is the totalization of the \u Cech complex of $\cF$ at our Nisnevich cover) and $\cF(X)\to\cF(X\times_S\bA^1)$ for any $\cF\in\PrSh_\textbf S(\cA_{/S})$ and any $X\in\cA_{/S}$.  This classification is often more useful when working with $\textbf H(\cA_{/S})$.
\end{defn}

We will let $\textbf H(S)$ denote the above category when $\cA_{/S}=\Sm_S$, or in other words $\cA=\Sm_B$ for our base scheme $B$.

It should be noted that $\textbf H(\cA_{/S})$ is symmetric monoidal via the cartesian product.

Now, we are almost at our desired definition.  In particular, we let $\textbf H_\bullet(\cA_{/S})$ refer to the pointed objects in $\textbf H(\cA_{/S})$.  The corresponding symmetric monoidal structure on $\textbf H_\bullet(\cA_{/S})$ is induced by $\wedge$.  The forgetful functor $\textbf H_\bullet(\cA_{/S})\to\textbf H(\cA_{/S})$ has a symmetric monoidal left adjoint given by taking a free disjoint baseboint $\cF\mapsto\cF_+$.

\begin{defn}
Given a vector bundle $\cE$ on $S$ with total space $E\to S$, we define the \emph{Thom space} $\text{Th}_S(\cE)$ to be the cofiber of the inclusion $E\setminus S\hookrightarrow E$.  Note that this is naturally an element of $\textbf H_\bullet(\cA_{/S})$ and that it must be compact.
\end{defn}

Denoting $\bT_S:=\text{Th}_S(\cO_S)=\bA^1/(\bA^1-S)$, we obtain suspension endofunctor
\[
\Sigma_\bT:=\bT_S\wedge(-):\textbf H_\bullet(\cA_{/S})\to\textbf H_\bullet(\cA_{/S}).
\]
This has a right adjoint loop space functor
\[
\Omega_\bT:\textbf H_\bullet(\cA_{/S})\to\textbf H_\bullet(\cA_{/S}).
\]
Now, we are ready to introduce our main definition.

\begin{defn}
Given $S\in\cS$, we may define the \emph{motivic stable homotopy category over $S$ relative to $\cA$}, denoted $\SH(\cA_{/S})$, to be the cofiltered homotopy colimit of
\[
\textbf H_\bullet(\cA_{/S})\overset{\Sigma_\bT}{\to}\textbf H_\bullet(\cA_{/S})\overset{\Sigma_\bT}{\to}\cdots
\]
taken in the $(\infty, 2)$-category of presentable $\infty$-categories with left-adjoint functors or equivalently as the filtered homotopy colimit of
\[
\cdots\overset{\Omega_\bT}{\to}\textbf H_\bullet(\cA_{/S})\overset{\Omega_\bT}{\to}\textbf H_\bullet(\cA_{/S})
\]
in either the $(\infty, 2)$-category of presentable $\infty$-categories with right adjoint functors, or just in the $(\infty, 2)$-category of $\infty$-categories.  An element of $\SH(\cA_{/S})$ will be referred to as an \emph{$\cA$-fibered spectrum over $S$}.
\end{defn}

This motivic stable homotopy category, originally defined by Voevodsky and Morel, has many nice properties, among them that it satisfies a six functors formalism.  When $\cA$ is not specified, as above, we assume that we are dealing with the subcategory of all smooth morphisms in $\cS$, and merely use the notation $\SH(S)$, referring to it as \emph{the motivic stable homotopy category over $S$}.

It should be noted that $\SH(\cA_{/S})$ has a natural symmetric monoidal structure, denoted $\otimes$, with the unit denoted $\mathbf1_S$.

\begin{thm}
The motivic stable homotopy category defines a presheaf of stable $\infty$-categories $\SH^*:\cS^{op}\to\Cat_\infty^{stab}$.  This presheaf naturally descends to the structure of a motivic $\infty$-category.
\end{thm}
\begin{proof}
This is essentially the entire first half of \cite{Khan} .
\end{proof}

For a deeper dive into the understanding of the $\infty$-categorical structure of the stable motivic homotopy category, the reader is encouraged to consult Marc Hoyois' phenomenal paper \cite{Hoy}, as well as \cite{Khan} .

An important category related to the motivic stable homotopy category is its rationalization, denoted $\SH_\Q(S):=\SH(S)\otimes\cD(\Q)$, where the tensor product is taken in the $\infty$-category of stable $\infty$-categories.  In addition to $\SH_\Q$ being a motivic $\infty$-category, it also satisfies several other nice properties that will be relevant for us in our upcoming examples.  Among these nice properties are the following.

\begin{thm}
    The rationalization of the motivic stable homotopy category $\SH_\Q$ satisfies the following properties:
    \begin{itemize}
        \item \emph{Absolute Purity:} for any morphism $f:X\to S$ of Noetherian schemes which is factorizable through a closed immersion and a smooth morphism, one obtains an equivalence
        \[
        \mathbf1_X\langle d\rangle\simeq f^!(\mathbf1_S)
        \]
        in $\SH_\Q(X)$, where $d=\rank(T_{X/S})$ is the virtual dimension of $f$
        \item \emph{Finiteness:} over quasi-excellent scheme, the six functors preserve constructibility
        \item \emph{Duality:} for every separated morphism of finite type $f:X\to S$ with $S$ quasi-excellent and regular, one has that $f^!(\mathbf1_S)$ is a dualizing object in $\SH_\Q(X)$
    \end{itemize}
\end{thm}
\begin{proof}
This is \cite{DFJK} corollary C parts II-IV.
\end{proof}

In addition, $\SH_\Q\simeq\textbf D_{\bA^1}(-, \Q)$, where $\textbf D_{\bA^1}(S, \Q)$ is the same stabilization applied to $\bA^1$-local complexes of sheaves of $\Q$-vector spaces satisfying Nisnevich descent.

One of our most important examples will be the category of Beilinson motives $\DM_B$, which is defined as the subcategory of modules over the motivic ring spectrum $\textbf H_B$ for varying $S$.  This will in particular be our primary example, and discussion of it will occupy much of a later section.  It should be noted that $\DM_B$ satisfies the same nice properties discussed above.

Before one gets the idea that all motivic categories comprise "categories of motives" or "categories of motivic spaces" of some form, let us introduce two related examples decidedly less motivic in nature.

\begin{defn}
Given some scheme $S\in\cS$, we define $\textbf D_{\acute et}(S, \Z/l\Z)$ to be the stable derived $\infty$-category of \'etale sheaves valued in $\Z/l\Z$ for some $l$ coprime to the exponential characteristic of our base scheme.  Further define $\textbf D(S, \Z_l)$ to be the stable derived $\infty$-category of $l$-adic sheaves on $S$.
\end{defn}

It has been known since SGA4 and the work of Deligne \cite{Del} that these two categories are motivic at least on the level of triangulated homotopy categories, and more recent work has demonstrated that both of them define motivic $\infty$-categories as well (see, for example, \cite{GaiRoz} or \cite{GaiLur}).

Note that one example we have NOT listed is that of Voevodsky motives.  It is still open whether or not Voevodsky's derived (stable $\infty-$)category of motives $\DM(S, \Q)$ over a scheme $S$ defines a motivic category in the sense we have described above.  It is certainly true if one restricts to particular choices of $\cS$.  For example, as we will discuss later, for any excellent geometrically unibranch base scheme $S$, one has that $\DM(S, \Q)\simeq\DM_B(S)$.  That said, it remains an open problem whether or not Voevodsky's category of motives satisfies a six functors formalism more generally.

\subsection{Constructing Derived Motivic Measures from Six Functors Formalisms}

The purpose of this section is to demonstrate the existence of a procedure for converting generalized six functors formalisms into derived motivic measures.  To do this, we will begin by proving identities about the four functors, and ultimately evaluate everything at the tensor unit over some base, yielding a weakly W-exact map that descends to our desired motivic measure.

\begin{prop}
Given any commutative triangle 
\[
\begin{tikzcd}
X \arrow[r, "f"] \arrow[dr, "h"]
    & Y \arrow[d, "g"]\\
&Z \end{tikzcd}
\]
of varieties, one has that the triangles
\[
\begin{tikzcd}
h_!h^! \arrow[r] \arrow[dr]
    & g_!g^! \arrow[d]\\
&\id \end{tikzcd}
\text{ and }
\begin{tikzcd}
\id \arrow[r] \arrow[dr]
    & g_*g^* \arrow[d]\\
&h_*h^* \end{tikzcd}
\]
commute.  The dual triangles for (co)units of the appropriate adjunctions also commute. 
\end{prop}
\begin{proof}
This statement is basically a specialization to the setting of the six functors formalism of the fact that $\infty$-categorical adjunctions compose (cf. \cite{RieVer}).
\end{proof}

\begin{cor}
Given a composition of closed immersions of $S$-schemes 
\[
\begin{tikzcd}
W \arrow[r, hook, "i"] \arrow[dr, "g"]
    & Z \arrow[d, "f"]\arrow[r, hook, "j"] & X \arrow[dl, "h"]\\
&S \end{tikzcd},
\]

And a composition of open immersions of $S$-schemes 
\[
\begin{tikzcd}
U \arrow[r, hook, "k"] \arrow[dr, "u"]
    & V \arrow[d, "s"]\arrow[r, hook, "l"] & Y \arrow[dl, "t"]\\
&S \end{tikzcd},
\]
one has the commutative triangles 
\[
\begin{tikzcd}
g_*g^! \arrow[r] \arrow[dr]
    & f_*f^! \arrow[d]\\
&h_*h^! \end{tikzcd}
\text{ and }
\begin{tikzcd}
 t_*t^!\arrow[r] \arrow[dr]
    & s_*s^! \arrow[d]\\
& u_*u^! \end{tikzcd}
\]
\end{cor}
\begin{proof}
This follows pretty directly from the preceding lemma.
\end{proof}

\begin{lem}
Suppose one has an adjunction of cospans of $\infty$-categories
\[
    \begin{tikzcd}
 W\arrow[d, "\psi"{name=F}, bend left=25]\arrow[r, "g"] & U\arrow[d, "\beta"{name=H}, bend left=25]& V\arrow{l}[swap]{f}\arrow[d, "\gamma"{name=J}, bend left=25]\\
  Z \arrow[u, "\phi"{name=G}, bend left=25]\arrow[phantom, from=F, to=G, "\dashv"]\arrow[r, "i"]& X\arrow[u, "\alpha"{name=I}, bend left=25]\arrow[phantom, from=H, to=I, "\dashv"]& \arrow{l}[swap]{h}Y\arrow[u, "\delta"{name=K}, bend left=25]\arrow[phantom, from=J, to=K, "\dashv"]
\end{tikzcd}
    \]
(note that the above diagram is not commutative as such; rather, it is commutative if one direction of vertical morphisms is omitted, but is presented above to highlight the adjointness of the different vertical morphisms).  Then the induced morphisms on pullbacks in both directions form an adjunction 
\[
\begin{tikzcd}
 Y\times_X Z \arrow[r, "\delta\times_\alpha\phi"{name=F}, bend left=25] & V\times_U W \arrow[l, "\gamma\times_\beta\psi"{name=G}, bend left=25]\arrow[phantom, from=F, to=G, "\dashv" rotate=-90]
\end{tikzcd}.
\]
\end{lem}
\begin{proof}
The crux of this proof will be to define the units and counits for our hypothesized adjunction and to demonstrate that they satisfy the triangle identities.  Let us take $\eta, \eta', \eta''$ and $\epsilon,\epsilon', \epsilon''$ to be the unit and counit of the middle, right, and left adjunctions, respectively.  Note that we get the isomorphisms 
\[
(\delta\times_\beta\psi)\circ(\gamma\times_\alpha\phi)\simeq(\delta\circ\gamma)\times_{(\beta\circ\alpha)}(\psi\circ\phi)
\]
and 
\[(\gamma\times_\alpha\phi)\circ(\delta\times_\beta\psi)\simeq(\gamma\circ\delta)\times_{(\alpha\circ\beta)}(\phi\circ\psi)\]
due to the natural commutativity of diagrams such as this one:
\[
    \begin{tikzcd}[row sep=1.5em, column sep = 1.5em]
    Y\times_XZ \arrow[rr, "\delta\times_\alpha\phi"] \arrow[dr, "p_1"] \arrow[dd, "p_2" near start] &&
    V\times_UW \arrow[dd, "p_2" near start] \arrow[dr, "p_1"] \arrow[rr, "\gamma\times_\beta\psi"] &&
    Y\times_XZ \arrow[dd, "p_2" near start] \arrow[dr, "p_1"]\\
    & Y \arrow[rr, "\delta" near start]\arrow[dd, "h" near start] &&
    V \arrow[dd, "f" near start] \arrow[rr, "\gamma" near start]&&
    Y \arrow[dd, "h" near start]\\
    Z \arrow[rr, "\phi" near start] \arrow[dr, "i"] && W\arrow[dr, "g"] \arrow[rr, "\psi" near start] && 
    Z \arrow[dr, "i"]\\
    & X \arrow[rr, "\alpha"] && U \arrow[rr, "\beta"] && X
    \end{tikzcd},
\]
and the result of what the values of composition along the long diagonal must be.  As a result, we obtain a natural unit and counit given by bringing $\eta'\times_\eta\eta''$ and $\epsilon'\times_\epsilon\epsilon''$ back along the equivalence.  All that remains is to show that these two satisfy the triangle identities, but this just follows directly from the triangle identities for each of the units and counits of the adjunctions in the cospan above.
\end{proof}

\begin{lem}
Suppose one has an adjunction of cospans of $\infty$-categories
\[
    \begin{tikzcd}
 C\arrow[d, "r"{name=F}, bend left=25]\arrow[r, "g"] & A\arrow[d, "p"{name=H}, bend left=25]& B\arrow{l}[swap]{f}\arrow[d, "q"{name=J}, bend left=25]\\
  F \arrow[u, "u"{name=G}, bend left=25]\arrow[phantom, from=F, to=G, "\dashv"]\arrow[r, "i"]& D\arrow[u, "s"{name=I}, bend left=25]\arrow[phantom, from=H, to=I, "\dashv"]& \arrow{l}[swap]{h}E\arrow[u, "t"{name=K}, bend left=25]\arrow[phantom, from=J, to=K, "\dashv"]
\end{tikzcd}
    \]
(as before, the above diagram is not commutative as such; rather, it is commutative if one direction of vertical morphisms is omitted, but is presented above to highlight the adjointness of the different vertical morphisms).  Then the induced morphisms on comma $\infty$-categories in both directions form an adjunction 
\[
\begin{tikzcd}
 \Hom_D(h, i) \arrow[r, " "{name=F}, bend left=25] & \Hom_A(f, g) \arrow[l, " "{name=G}, bend left=25]\arrow[phantom, from=F, to=G, "\dashv" rotate=-90]
\end{tikzcd}.
\]
\end{lem}
\begin{proof}
The adjunction of cospans in the lemma description yields a related adjunction of cospans
\[
    \begin{tikzcd}
 C\times B\arrow[d, "r\times q"{name=F}, bend left=25]\arrow[r, "g\times f"] & A\times A\arrow[d, "p\times p"{name=H}, bend left=25]& A^{\Delta[1]}\arrow{l}[swap]{(p_1, p_0)}\arrow[d, "p^{\Delta[1]}"{name=J}, bend left=25]\\
  F\times E \arrow[u, "u\times t"{name=G}, bend left=25]\arrow[phantom, from=F, to=G, "\dashv"]\arrow[r, "i\times h"]& D\times D\arrow[u, "s\times s"{name=I}, bend left=25]\arrow[phantom, from=H, to=I, "\dashv"]& \arrow{l}[swap]{(p_1, p_0)} D^{\Delta[1]} \arrow[u, "s^{\Delta[1]}"{name=K}, bend left=25]\arrow[phantom, from=J, to=K, "\dashv"].
\end{tikzcd}
    \]
    
    Now, note that by definition, one has that the pullbacks of these cospans are the appropriate comma $\infty$-categories, so we obtain our desired adjunction.
\end{proof}

\begin{prop}
Suppose one has a cartesian square of schemes of the form
 \[
    \begin{tikzcd}
 X\arrow[d, hook, "i"]\arrow[r, hook, "j"]\arrow[dr, phantom, "\lrcorner", very near start] & Y \arrow[d, hook, "i'"]\\
  Z \arrow[r, hook, "j'"]& W
\end{tikzcd},
    \]
    with $i$ and $i'$ closed immersions and $j$ and $j'$ open immersions.  Then one has that the triangles
    \[
     \begin{tikzcd}[column sep=small]
  i'_!i'^!j'_*j'^* \arrow{dr}[swap]{\epsilon_{i'}^!}\arrow[rr, no head, "\sim"] &                         & j'_*i_!i^!j'^*\arrow{dl}{\epsilon_i^!}\\
  & j'_*j'^*  & 
\end{tikzcd}
    \text{ and }
     \begin{tikzcd}[column sep=small]
& i'_!i'^! \arrow{dl}[swap]{\eta_j^*} \arrow{dr}{\eta_{j'}^*} & \\
  i'_!j_*j^*i'^! \arrow[rr, no head, "\sim"] &                         & j'_*j'^*i'_!i'^!
\end{tikzcd}
    \]
    homotopy commute.
\end{prop}
\begin{proof}
This proposition can be proven by showing that it is true for each relevant component of the units and counits involved.  We start with the shriek units.  Note that given any $\cF\in\bD(W)$, we have the following adjunction of cospans of $\infty$-categories
\[
    \begin{tikzcd}
 \bD(X)\arrow[d, "j_*"{name=F}, bend left=25]\arrow[r, "i_!\simeq i_*"] & \bD(Z)\arrow[d, "j'_*"{name=H}, bend left=25]& \textbf{1}\arrow{l}[swap]{j'^*\cF}\arrow[d, equal]\\
  \bD(Y) \arrow[u, "j^*"{name=G}, bend left=25]\arrow[phantom, from=F, to=G, "\dashv"]\arrow[r, "i'_!\simeq i'_*"]& \bD(W)\arrow[u, "j'^*"{name=I}, bend left=25]\arrow[phantom, from=H, to=I, "\dashv"]& \arrow{l}[swap]{j'_*j'^*\cF}\textbf{1}
\end{tikzcd},
    \]
    where the commutativity of the upwards-oriented lefthand square is the result of base change, and the commutativity of the upwards-oriented righthand square is the result of the triangle equivalence (the downwards-oriented squares are more immediately commutative).  This implies that the induced functor $\Hom_{\bD(Z)}(i_!, j'^*\cF)\to\Hom_{\bD(W)}(i'_!, j'_*j'^*\cF)$ is a right adjoint, and thus preserves limits.  In particular, it preserves terminal objects, which shows that for any $\cF\in\bD(W)$, the first triangle commutes, and thus the triangle commutes in the category of functors.  The second proof is dual to the first.
\end{proof}

\begin{lem}
Suppose that we have $\infty$-functors of the form $F, G:\cB\to\cC$ and $H, I:\cA\to\cB$ equipped with natural transformations $F\overset{\alpha}{\to}H$ and $G\overset{\beta}{\to}I$.  Then we have the commutative square
\[
\begin{tikzcd}
 F\circ G\arrow[r]\arrow[d] & F\circ I\arrow[d] \\
 H\circ G\arrow[r] & H\circ I
\end{tikzcd}.
\]
\end{lem}
\begin{proof}
This is just a restatement of the fact that $2$-morphisms compose horizontally in the $(\infty, 2)$-category of $\infty$-categories.
\end{proof}

\begin{lem}
Suppose we are in the situation of the last lemma, except that $\cA=\cB=\cC$, and all of our $\infty$-functors are $\infty$-autofunctors.  Let us further assume that all four of the endofunctors homotopy commute in a way that is compatible in the homotopy category.  Then we have a commutative cube
\[
    \begin{tikzcd}[row sep=1.5em, column sep = 1.5em]
    F\circ G \arrow[rr] \arrow[dr] \arrow[dd] &&
    F\circ I \arrow[dd] \arrow[dr] \\
    & H\circ G \arrow[rr]\arrow[dd] &&
    H\circ I \arrow[dd] \\
    G\circ F \arrow[rr] \arrow[dr] && I\circ F\arrow[dr] \\
    & G\circ H \arrow[rr] && I\circ H
    \end{tikzcd},
\]
where the vertical morphisms are simply the homotopy-compatible morphisms.  Furthermore, considered as an $\infty$-morphism of commutative squares, the diagram above is invertible.
\end{lem}
\begin{proof}
This proof will make heavy use of the fact that given an $\infty$-category $\cC$, the functor $\Ho(\cC^{\Delta[1]})\to\Ho(\cC)^{\Delta[1]}$ is surjective on objects, full, and conservative (\cite{RieVer} 3.1.1).  In particular, this further implies that for any $n$, $\Ho(\cC^{\Delta[1]^n})\to\Ho(\cC)^{\Delta[1]^n}$ is surjective on objects, full, and conservative as well.

Now, consider the category of cubes $\cC^{\Delta[1]^3}\simeq(\cC^{\Delta[1]})^{\Delta[1]^2}\simeq(\cC^{\Delta[1]^2})^{\Delta[1]}$.  The primary characterization of cubes that we will make use of is as arrows in commutative squares, or the middle term in the above string of equivalences.  By our assumptions, we see that the cube written in the lemma description above commutes when considered in the homotopy category $(\Ho(\cC)^{\Delta[1]^2})^{\Delta[1]}$.  Since the canonical functor $\Ho(\cC^{\Delta[1]^3})\to\Ho(\cC)^{\Delta[1]^3}$ is surjective on objects, one has that there must exist a commutative diagram $D\in\Ho(\cC^{\Delta[1]^3})$ (in other words, an object of $\cC^{\Delta[1]^3}$) filling the commutative cube, which settles the first part of the lemma.

Let $J$ denote the interval groupoid and consider an inclusion $\Delta[1]\hookrightarrow J$.  This induces for any $\infty$-category $\cC$ a map $\cC^J\to\cC^{\Delta[1]}$ whose essential image is the subcategory of equivalences in $\cC$. Furthermore, the map from $\cC^J$ onto its essential image is essentially surjective, as a morphism in an $\infty$-category is an equivalence if and only if it can be lifted to a map from $J$ to $\cC$.  (Indeed, it is surjective on objects, in spite of this not being a homotopy-invariant notion.)  In particular, on $1$-categories, this is the inclusion of a subcategory since inverses to equivalences are uniquely defined in $1$-categories.  Now, we have the commutative square 
\[
\begin{tikzcd}
 \Ho(\cC^{\Delta[1]})\arrow[r] & \Ho(\cC)^{\Delta[1]}\\
 \Ho(\cC^J)\arrow[u]\arrow[r] & \Ho(\cC)^J\arrow[u]
\end{tikzcd}.
\]
The bottom morphism is conservative by [R-V].  We wish to show that it is surjective on objects as well.  Note that both the left and right vertical morphism are surjective on objects.  Note that the vertical morphisms precisely describe those morphisms, either in the underlying $\infty$-category or in the homotopy category, which are invertible.  Furthermore, the map restricted to these subcategories must also be surjective on objects.  Hence, one must have that 
\[
\Ho(\cC^J)\to\Ho(\cC)^J
\]
is surjective on objects as well.  Now, we note that there is a string of morphisms
\[
\Ho((\cC^J)^{\Delta[1]^2})\to\Ho(\cC^J)^{\Delta[1]^2}\to(\Ho(\cC)^J)^{\Delta[1]^2}
\]
which are surjective on objects.  Hence, it must be the case that our cube above admits a filling to $(\cC^J)^{\Delta[1]^2}$.  Furthermore, since $(\cC^J)^{\Delta[1]^2}\simeq(\cC^{\Delta[1]^2})^J$, our cube fills to an equivalence of commutative squares in $\cC$.  Finally, this filling must be unique up to a contractible choice since all of this is done ambiently in the $(\infty, 2)$-category of $\infty$-categories.
\end{proof}

\begin{prop}
Suppose one has a cartesian square of schemes of the form
 \[
    \begin{tikzcd}
 X\arrow[d, hook, "i"]\arrow[r, hook, "j"]\arrow[dr, phantom, "\lrcorner", very near start] & Y \arrow[d, hook, "i'"]\\
  Z \arrow[r, hook, "j'"]& W
\end{tikzcd},
    \]
    with $i$ and $i'$ closed immersions and $j$ and $j'$ open immersions.  Then the square 
    \[
    \begin{tikzcd}
 j'_*i_!i^!j'^*\simeq i'_!j_*j^*i'^!\arrow{d}[swap]{\epsilon_i^!}& i'_!i'^! \arrow{l}[swap]{\eta_j^*} \arrow[d, "\epsilon_{i'}^!"]\\
  j'_*j'^* & \id\arrow[l, "\eta_{j'}^*"]
\end{tikzcd}
    \]
    in $\bD(W)$ commutes.
\end{prop}
\begin{proof}
Let us begin by noting that due to the various forms of base change, we have the following string of equivalences in $\bD(W)$ which we can employ:
\[
j'_*i_!i^!j'^*\simeq i'_!j_*j^*i'^!\simeq i'_!i'^!j'_*j'^*\simeq j'_*j'^*i'_!i'^!.
\]
This, combined with the commutativity of the triangles of proposition 3.29 allows us to boil the proof down to noting the commutativity of the squares
\[
\begin{tikzcd}
 i'_!i'^!j'_*j'^*\arrow{d}[swap]{\epsilon_{i'}^!}& i'_!i'^! \arrow{l}[swap]{\eta_j^*} \arrow[d, "\epsilon_{i'}^!"]\\
  j'_*j'^* & \id\arrow[l, "\eta_{j'}^*"]
\end{tikzcd}
\text{ and }
\begin{tikzcd}
 j'_*j'^*i'_!i'^!\arrow{d}[swap]{\epsilon_i^!}& i'_!i'^! \arrow{l}[swap]{\eta_j^*} \arrow[d, "\epsilon_{i'}^!"]\\
  j'_*j'^* & \id\arrow[l, "\eta_{j'}^*"]
\end{tikzcd}
\]
by lemma 3.30 and then demonstrating that these two squares are connected by isomorphisms which yield a commutative cube.  This last step is precisely what we are left with at this point.  
Indeed, we have the cube
\[
    \begin{tikzcd}[row sep=1.5em, column sep = 1.5em]
    i'_!i'^!\circ\id \arrow[rr] \arrow[dr] \arrow[dd] &&
    i'_!i'^!\circ j'_*j'^* \arrow[dd] \arrow[dr] \\
    & \id\circ\id \arrow[rr]\arrow[dd, equal] &&
    \id\circ j'_*j'^* \arrow[dd] \\
    \id\circ i'_!i'^! \arrow[rr] \arrow[dr] && j'_*j'^*\circ i'_!i'^!\arrow[dr] \\
    & \id\circ\id \arrow[rr] && j'_*j'^*\circ \id
    \end{tikzcd},
\]
which must be (homotopy) commutative by the preceding lemma.
\end{proof}
\begin{cor}
Given a cartesian square of $S$-schemes
\[
    \begin{tikzcd}
 X\arrow[dd, hook, "i"]\arrow[rr, "j"] \arrow[dr, "f_X"]& & Y \arrow[dd, hook , "i'"]\arrow[dl, "f_Y"]\\
& S &\\
  Z \arrow[rr, "j'"]\arrow[ur, "f_Z"] & & W \arrow[ul, "f_W"]
\end{tikzcd}
\]
where the horizontal morphisms are open immersions and the vertical morphisms are closed immersions, one has that the square
\[
    \begin{tikzcd}
 {f_X}_*f_X^!\arrow{d}[swap]{\epsilon_i^!}& {f_Y}_*f_Y^! \arrow{l}[swap]{\eta_j^*} \arrow[d, "\epsilon_{i'}^!"]\\
  {f_Z}_*f_Z^! & {f_W}_*f_W^!\arrow[l, "\eta_{j'}^*"]
\end{tikzcd}
    \]
    in $\bD(S)$ commutes.
\end{cor}
\begin{proof}
This is simply the commutative square of the preceding lemma after precomposition with $f_W^!$ and postcomposition with ${f_W}_*$.
\end{proof}

\begin{prop}
Given a commutative diagram of finite type $S$-schemes
\[
\begin{tikzcd}
Z \arrow[r, hook, "i"] \arrow[dr, "g"]
    & X \arrow[d, "f"] & U \arrow{l}{\circ}[swap]{j} \arrow[dl, "h"]\\
&S \end{tikzcd},
\]
where $i$ is a closed immersion, and $j$ is its complementary open immersion, one obtains a natural cofiber sequence
\[
g_*g^!\to f_*f^!\to h_*h^!.
\]
\end{prop}
\begin{proof}
Consider a commutative diagram of finite type $S$-schemes (where the base $S$ is also of finite type)
\[
\begin{tikzcd}
Z \arrow[r, hook, "i"] \arrow[dr, "g"]
    & X \arrow[d, "f"] & U \arrow{l}{\circ}[swap]{j} \arrow[dl, "h"]\\
&S \end{tikzcd},
\]
where $i$ is a closed immersion and $j$ is its complementary open immersion.  Note that since $\bD$ satisfies the six functors formalism, in particular, it satisfies the localization property.  Thus, considering $Z\overset{i}{\hookrightarrow} X\overset{j}{\underset{\circ}{\leftarrow}} U$ as absolute finite type schemes, one has the natural cofiber sequence $i_!i^!\to\id\to j_*j^*$.  

Now, since $i$ is a closed immersion, it is in particular proper, so we have that $i_!\simeq i_*$.  Furthermore, since $j$ is an open immersion, it is smooth, so one has that $j^*\simeq\Sigma^{\Omega_j}\circ j^!$.  Since $j$ is \'{e}tale, $\Omega_j=0$, so $\Sigma^{\Omega_j}$ is equivalent to the identity, and we are left with $j^*\simeq j^!$.

Composing with the appropriate direction of the above equivalences, we get the cofiber sequence 
\[
i_*i^!\to\id\to j_*j^!.
\]
Now, precomposing with $f^!$, we obtain a cofiber sequence 
\[
i_*i^!f^!\to f^!\to j_*j^!f^!.
\]
Since $f_*$ preserves finite limits and any (co)cartesian square in a stable $\infty$-category is bicartesian, we have that postcomposition with $f_*$ yields a cofiber sequence 
\[
f_*i_*i^!f^!\to f_*f^!\to f_*j_*j^!f^!.
\]
Finally, noting that $f_*i_*i^!f^!\simeq (f\circ i)_*(f\circ i)^!=g_*g^!$ and that $f_*j_*j^!f^!\simeq (f\circ j)_*(f\circ j)^!=h_*h^!$, this reduces to the cofiber sequence
\[
g_*g^!\to f_*f^!\to h_*h^!.
\]
\end{proof}

\begin{thm}
Suppose that $\bD$ is constructibly generated and satisfies one of the following two sets of conditions:
\begin{itemize}
    \item The four functors preserve constructible objects when given input a seperable morphism of finite type (note that compactness is trivially preserved by tensor)
    \item The six functors preserve constructible objects over Noetherian quasi-excellent schemes  of finite dimension with respect to morphisms of finite type (in other words, for any finite type morphism $f:X\to S$ with target Noetherian quasi-excellent of finite dimension, the four functors preserve constructible objects, while tensor and $\Hom$ preserve constructible objects over $S$)
\end{itemize}  
Then, given a scheme $S$ (assumed to be Noetherian quasi-excellent of finite dimension if $\bD$ satisfies the second condition in particular), there is a weakly W-exact functor \[M^c_{\bD(S)}:\Var_S\to\bD_{\cons}(S)\]
sending each variety (smooth or otherwise) $(X\overset{f}{\to}S)\in\Var_S$ to $M^c_{\bD(S)}(X):=f_*f^!\mathbf{1}_S$.
\end{thm}

\begin{proof}
Assembling the various lemmae and propositions that we have proven above, we are equipped to show the following:
\begin{itemize}
    \item We have a covariant functor $M^c_{\bD(S)}:\cof(\Var_S)\to\bD_{\cons}(S)$ given by sending objects $X$ to $M^c_{\bD(S)}(X)$, and morphisms $Z\overset{i}{\hookrightarrow}X$ to $M^c_{\bD(S)}(Z)\overset{\epsilon_i^!(\mathbf1_S)}{\longrightarrow}M^c_{\bD(S)}(X)$, where functoriality comes from evaluating the left triangle in the statement of corollary 3.26 at $\mathbf{1}_S$
    \item We have a contravariant functor $M^c_{\bD(S)}:\comp(\Var_S)\to\bD_{\cons}(S)$ given by sending objects $X$ to $M^c_{\bD(S)}(X)$, and morphisms $U\overset{j}{\hookrightarrow}X$ to $M^c_{\bD(S)}(X)\overset{\eta_j^*(\mathbf1_S)}{\longrightarrow}M^c_{\bD(S)}(U)$, where functoriality comes from evaluating the right triangle in the statement of corollary 3.26 at $\mathbf{1}_S$
    \item Since weak equivalences in $\Var_S$ are simply isomorphisms, and isomorphisms are in particular closed immersions, we obtain our functor $M^c_{\bD(S)}:w(\Var_S)\to\bD_{\cons}(S)$ by restricting the one we already have on closed immersions.  Note that isomorphisms must map to weak equivalences in $\bD_{\cons}(S)$
    \item Consider a cartesian square of $S$-motives
    \[
    \begin{tikzcd}
    X\arrow[dd, hook, "i"]\arrow[rr, "j"] \arrow[dr, "f_X"]& & Y \arrow[dd, hook, "i'"]\arrow[dl, "f_Y"]\\
    & S &\\
      Z \arrow[rr, "j'"]\arrow[ur, "f_Z"] & & W \arrow[ul, "f_W"]
    \end{tikzcd}
    \]
    where the horizontal morphisms are open immersions and the vertical morphisms are closed immersions.  By corollary 3.33, we have the commutative square 
    \[
    \begin{tikzcd}
    {f_X}_*f_X^!\arrow{d}[swap]{\epsilon_i^!}& {f_Y}_*f_Y^! \arrow{l}[swap]{\eta_j^*} \arrow[d, "\epsilon_{i'}^!"]\\
     {f_Z}_*f_Z^! & {f_W}_*f_W^!\arrow[l, "\eta_{j'}^*"]
    \end{tikzcd},
    \]
    Which, upon evaluation at $\mathbf1_S$, in turn yields the commutative square
    \[
    \begin{tikzcd}
    M^c_{\bD(S)}(X)\arrow{d}[swap]{\epsilon_i^!}& M^c_{\bD(S)}(Y) \arrow{l}[swap]{\eta_j^*} \arrow[d, "\epsilon_{i'}^!"]\\
     M^c_{\bD(S)}(Z) & M^c_{\bD(S)}(W)\arrow[l, "\eta_{j'}^*"]
    \end{tikzcd}
    \]
    \item Given a closed/open decomposition
    \[
    \begin{tikzcd}
    Z \arrow[r, hook, "i"] \arrow[dr, "g"]
    & X \arrow[d, "f"] & U \arrow{l}{\circ}[swap]{j} \arrow[dl, "h"]\\
    &S \end{tikzcd}
    \]
    of an $S$-variety $X$, composing the natural cofiber sequence 
    \[
    g_*g^!\to f_*f^!\to h_*h^!
    \]
    of proposition 3.34 with $\mathbf1_S$ yields the cofiber sequence
    \[
    M^c_{\bD(S)}(Z)\to M^c_{\bD(S)}(X)\to M^c_{\bD(S)}(U)
    \]
    \item Of the last two conditions we need to check, the one on cofibrations follows from the fact that the two compositions around the diagram of $S$-varieties
     \[
    \begin{tikzcd}
    X\arrow[dd, hook, "i"]\arrow[rr, "\sim"] \arrow[dr, "f_X"]& & Y \arrow[dd, hook, "i'"]\arrow[dl, "f_Y"]\\
    & S &\\
      Z \arrow[rr, "\sim"]\arrow[ur, "f_Z"] & & W \arrow[ul, "f_W"]
    \end{tikzcd}
    \]
    (where the vertical morphisms are closed immersions and the horizonal morphisms are isomorphisms) coincide, and by the left triangle of corollary 3.26.  The condition on complements is a special case of the fourth bullet point of this proof, as all isomorphisms are necessarily cofibrations.
    
\end{itemize}
\end{proof}
\begin{cor}
Suppose $\bD$ sastisfies one of the two conditions of the above theorem.  Then, given a scheme $S$ (assumed to be Noetherian quasi-excellent of finite dimension if $\bD$ satisfies the second set of conditions), one obtains a map of K-theory spectra 
\[K(M^c_{\bD(S)}):K(\Var_S)\to K(\bD_{\cons}(S)).\]
\end{cor}
\begin{proof}
This follows immediately from the existence of the weakly W-exact functor above.
\end{proof}
Ultimately, the reason we wanted to ensure that we landed in constructible objects above was simply to avoid swindles that would be permitted by mapping into an essentially large category.  The fact that our category is essentially small ensures that our K-theory is nontrivial.  Note that the only condition specified above is that the four functors preserve constructibility.  While it is certainly preferable that $\bD$ be compactly or (better yet) constructibly generated, it is not strictly speaking needed to have a well-defined W-exact functor/map on K-theory of the type above.

\section{Lifting the Gillet-Soul\'{e} Motivic Measure}
\subsection{Preamble on the Gillet-Soul\'e Motivic Measure and Some Work of Bondarko}

We begin this section with a brief overview of the classical Gillet-Soul\'e motivic measure, as well as a slight generalization due to Gillet and Soul\'e, before briefly describing a few theorems of Bondarko that allow for an alternate characterization of the Gillet-Soul\'e motivic measure in terms of Voevodsky motives.  This will then be used in the following sections to construct a derived motivic measure lifting the Gillet-Soul\'e motivic measure by passing through Voevodsky motives.

Given a field $k$ which satisfies resolution of singularities and weak-factorization, we have the following theorem due to Bittner (see \cite{Bitt} theorem 3.1 or \cite{MNP} theorem 9.1.2).

\begin{thm}
If $k$ is a field which admits resolution of singularities and weak-factorization, then $K_0(\Var_k)$ may be presented by the isomorphism classes $[X]$ of smooth projective varieties over $k$ subject to the relations
\begin{itemize}
    \item $[\emptyset]=0$
    \item $[X]-[Z]=[Y]-[E]$ where $Z\subseteq X$ is a closed subvariety, $Y=\text{Bl}_Z(X)$, and $E$ is the exceptional divisor
\end{itemize}
\end{thm}

As a result, letting $h:\textbf{SmProj}_k\to\text{Chow}(k, \Q)$ be the natural map from smooth projective varieties to Chow motives, one can make the following definition:

\begin{defn}
If $k$ satisfies resolution of singularities and weak factorization, then the map 
\[
\chi_{gs}:K_0(\Var_k)\to K_0(\text{Chow}(k, \Q))
\]
given by $\chi_{gs}([X]):=[h(X)]$ is referred to as the \emph{Gillet-Soul\'e motivic measure}.
\end{defn}

Now, this definition is really also a proposition (\cite{MNP} proposition 9.1.3), as it is nontrivial that one must have $[h(X)]-[h(Z)]=[h(Y)]-[h(E)]$ as above.  That said, we suppress the proof here for brevity.

This motivic measure can actually be redefined in such a way as to employ the standard generators and relations.  Recall that given any pseudo-abelian category $\mathfrak A$, it is always possible not only to define the category $\mathbf{CH}^\flat\mathfrak A$ of bounded chain complexes, but also its chain homotopy category $\text{Hot}^\flat\mathfrak A$.  In particular, one has that 
\[K_0(\text{Hot}^\flat\mathfrak A)\cong K_0(\mathfrak A)\]
via the Euler characteristic map $[A_\bullet]\mapsto\sum_i(-1)^i[A_i]$ via \cite{GilSou} lemma 3. Using the fact that, in particular, $K_0(\Chow(k, \Q)\simeq K_0(\text{Hot}^\flat\Chow(k, \Q))$, we can immediately come up with a refinement of the Gillet-Soul\'e motivic measure as well as a categorification which anticipates the weakly W-exact functors we will construct below.  Note the following definition/theorem.

\begin{thm}
Given any arbitrary $X\in\Var_k$, we may construct a complex $W(X)\in\text{Hot}^\flat\Chow(k, \Q)$ which refines the Gillet-Soul\'e motivic measure in the sense that $[W(X)]\mapsto[h(X)]\in K_0(\Var_k)$ whenever $X$ is smooth and projective.  This is referred to as the \emph{weight complex}.  The assignment $X\mapsto W(X)$ satisfies the following functoriality properties:
\begin{itemize}
    \item A proper map $f:X\to Y$ induces a map $f^*:W(Y)\to W(X)$ and for any two composable proper maps $f$ and $g$, one has that $(f\circ g)^*=g^*\circ f^*$.  In other words, there is a contravariant functor from varieties equipped with proper morphisms to complexes of Chow motives up to homotopy
    \item An open immersion $j:U\overset{\circ}{\to}X$ induces a map $j_*: W(U)\to W(X)$ which is covariantly functorial in open immersions analogously to the above
    \item For any $X, Y\in\Var_k$, one has that $W(X\times Y)\cong W(X)\otimes W(Y)$
    \item Given any closed-open decomposition $Z\overset{i}{\hookrightarrow}X\overset{j}{\underset{\circ}{\leftarrow}}U$ of a $k$-variety $X$, one obtains the exact triangle $W(U)\overset{j_i}{\to}W(X)\overset{i^*}{\to}W(Z)\to W(U)[1]$ in $\text{Hot}^\flat\Chow(k, \Q)$
\end{itemize}
\end{thm}
\begin{proof}
This is \cite{GilSou} theorem 2 (also \cite{MNP} theorem 9.2.1).
\end{proof}
Given the definition of the Grothendieck group of a triangulated category, we note that the last property of the weight complex functors provides a categorical lift of the cutting-and-pasting property enjoyed by the Gillet-Soul\'e motivic measure.  We will not go into the explicit construction of $W(X)$ here, as it will serve more as a bridge between Voevodsky motives and Chow motives and will not be used in and of itself.  That said, for those who are interested in the construction, we recommend the phenomenal paper by Gillet and Soul\'e \cite{GilSou} and the book by Murre \cite{MNP} .

Note that this construction does not work for us as such, given that the maps constructed "run in the wrong direction," among other things.  That said, this can be rectified by simply reversing the arrows and working with homological Chow motives as is described in \cite{Bon} remark 6.5.2 instead of cohomological ones.  This will not alter the underlying Grothendieck ring, so from this point forward, we will work exclusively with the homological grading.  The paper by Bondarko cited above is very deep and has much more to offer than simply what is taken from it here; in what follows, we will only describe the bare minimum of what we need.  Most importantly:

\begin{thm}
There exists a conservative functor $t:\dm_{gm}^{eff}(k, \Q)\to\text{Hot}^\flat\Chow^{eff}(k, \Q)$ called the \emph{weight complex functor} (for reasons we shall see shortly) which induces an isomorphism of Grothendieck rings $K_0(\dm_{gm}^{eff}(k, \Q))\overset{\sim}{\to}K_0(\text{Hot}^\flat\Chow^{eff}(k, \Q))$.  This descends on $K_0$ to an isomorphism $K_0(\dm_{gm}(k, \Q))\overset{\sim}{\to}K_0(\text{Hot}^\flat\Chow(k, \Q))$.
\end{thm}
\begin{proof}
This is \cite{Bon} proposition 6.3.1 combined with remark 6.3.2.
\end{proof}

\begin{thm}
For any $X\in\Var_k$, we have that $t(M^c(X))\cong W(X)$, and that this assignment is functorial when we restrict to the category of varieties equipped with proper maps.
\end{thm}
\begin{proof}
This is essentially \cite{Bon} proposition 6.6.2.  Our notation difference from the aforementioned proposition is explicable via remark 6.3.2 (2) of the same paper.
\end{proof}

Thus, not only is $K_0(\dm_{gm}(k, \Q))\cong K_0(\Chow(k, \Q))$, but also the corresponding classes of $M^c(X)$ and $W(X)$ always coincide.  Indeed, we have reduced our task to lifting a stable or DG-enhancement of the assignment $X\mapsto M^c(X)$ to a weakly W-exact functor, as doing so will provide a lift of the Gillet-Soul\'e motivic measure.  Two different approaches are given in the section below.

\subsection{Approach via Six Functors}

For the entirety of this section, all schemes will be based over a perfect field $k$, and we will be working exclusively with rational coefficients (more general coefficients will perhaps be addressed in forthcoming work).

Consider the motivic spectrum $KGL_S\in\textbf{SH}(S)$ defined to be the representative of algebraic K-theory in $\textbf{SH}(S)$.  It should be noted that for any morphism of schemes $f:X\to Y$, one has that $f^*KGL_Y\simeq KGL_X$. Furthermore, its rationalization $KGL_{S, \Q}:=KGL_S\otimes\Q$ breaks up as the sum
\[
KGL_{S, \Q}\simeq \bigoplus_{i\in\Z}KGL^{(i)}_S
\]
in a way that is compatible with base change.

\begin{defn}
We define the Beilinson motivic cohomology over $S$ to be
\[
H_{B, S}:=KGL^{(0)}_S
\]
and define the stable $\infty$-category of Beilinson motives over $S$ to be 
\[
\textbf{DM}_B(S):=\Mod_{H_{B, S}}.
\]
\end{defn}

\begin{thm}
$\textbf{DM}_B$ admits a constructibly generated six functors formalism.  In particular, $\textbf{DM}_B^*$ defines a $(*, \#, \otimes)$-formalism in the sense of Khan, satisfies the Voevodsky conditions, and is constructibly generated.
\end{thm}
\begin{proof}
This is part of \cite{RicSch} synopsis 2.1.1 based on prior work by Cisinski and D\'eglise in \cite{CisDeg} and Ayoub in \cite{Ayo} .
\end{proof}

\begin{thm}
    The motivic category of Beilinson motives $\DM_B$ satisfies the following properties:
    \begin{itemize}
        \item \emph{Absolute Purity:} for any morphism $f:X\to S$ of Noetherian schemes which is factorizable through a closed immersion and a smooth morphism, one obtains an equivalence
        \[
        \mathbf1_X\langle d\rangle\simeq f^!(\mathbf1_S)
        \]
        in $\SH_\Q(X)$, where $d=\rank(T_{X/S})$ is the virtual dimension of $f$
        \item \emph{Finiteness:} over quasi-excellent scheme, the six functors preserve constructibility
        \item \emph{Duality:} for every separated morphism of finite type $f:X\to S$ with $S$ quasi-excellent and regular, one has that $f^!(\mathbf1_S)$ is a dualizing object in $\SH_\Q(X)$
    \end{itemize}
\end{thm}
\begin{proof}
This is \cite{DFJK} theorem A parts II-IV under the equivalence found in part V of the same theorem (based on prior work by Cisinski and D\'eglise in \cite{CisDeg}).
\end{proof}

We also obtain the following corollary.

\begin{cor}
 Given $f:X\to Y$ separated of finite type with $X$ and $Y$ quasi-excellent, the involutive antiequivalence $D_X:=\Hom_X(-, f^!\mathbf{1}_Y)$ on the category $\DM_B(X)$ descends to one on the category of compact/constructible objects $\DM_B^c(X)$.  
\end{cor}

\begin{proof}
This combines the second and third statements of the preceding theorem.
\end{proof}

\begin{prop}
The involution $D_{(-)}$ intertwines the four functors in the following way.  Where appropriately defined, if one has a morphism $g: S\to T$, one has that
 \[
 D_T\circ g_!\simeq g_*\circ D_S\text{ and }g^*\circ D_T\simeq D_S\circ g^!.
 \]
\end{prop}
\begin{proof}
This is \cite{RicSch} synopsis 2.1.1 part viii.
\end{proof}

The reason that we mostly use the above approach is that it has an extremely rigid structure and many nice properties, including but not limited to absolute purity.  As an example, the above theorems yield the following central result as a corollary.

\begin{cor}
Given any Noetherian, quasi-excellent scheme $S$ of finite dimension, one obtains a weakly W-exact functor 
\[M^c_S:\Var_S\to\DM_B^c(S)\]
which descends on K-theory to a map of K-theory spectra
\[K(M^c_S):K(\Var_S)\to K(\DM_B^c(S)).\]
\end{cor}
\begin{proof}
This follows directly from the above theorem.
\end{proof}

In particular, this yields a generalized derived Gillet-Soul\'{e} motivic measure, as we shall see shortly.  Before we can say this with certainty, we must take a detour through the DG category of Voevodsky motives.

\begin{cor}
Over any excellent, geometrically unibranch scheme $S$, the one obtains an adjunction
\[
\begin{tikzcd}
\text{SH}_\Q(S)\ar[r,bend left," ",""{name=A, below}] & \text{DM}(S, \Q)\ar[l,bend left," ",""{name=B,above}] \ar[from=A, to=B, phantom, "\dashv" rotate=270]
\end{tikzcd}
\]
on homotopy categories where the left adjoint arises as sheafification with respect to transfers and the right adjoint arises as a forgetful functor (forgetting transfers).  This descends to an equivalence
\[
\begin{tikzcd}
\text{DM}_B(S)\ar[r,bend left," ",""{name=A, below}] & \text{DM}(S, \Q)\ar[l,bend left," ",""{name=B,above}] \ar[from=A, to=B, phantom, "\dashv" rotate=270]
\end{tikzcd}
\]
which further restricts to an equivalence
\[
\begin{tikzcd}
\text{DM}_B^c(S)\ar[r,bend left," ",""{name=A, below}] & \text{DM}_{gm}(S, \Q)\ar[l,bend left," ",""{name=B,above}] \ar[from=A, to=B, phantom, "\dashv" rotate=270]
\end{tikzcd}
\]
on the level of compact/constructible objects.
\end{cor}

\begin{proof}
This is \cite{CisDeg} theorem 16.1.4 coupled with noting that geometric motives are simply the compact objects in Voevodsky's big category of motives.
\end{proof}

\begin{prop}
Specializing the above equivalence to the case of $S=\Spec{(k)}$, one has that for all $k$-schemes $f:X\to \Spec(k)$, considering $\mathbf{1}_k\in\text{DM}_B(k)$ one has 
\[
f_*f^!(\mathbf{1}_k)\mapsto M^c(X),
\]
the compactly supported motive of X (thus justifying our notation above).
\end{prop}
\begin{proof}
Note first of all that via \cite{CisDeg} theorem 16.1.4 and \cite{CisDeg2} corollary 5.9, the sheafification functors
\[
\dm_B(k)\to\dm(k, \Q)\to\dm_{cdh}(k, \Q)
\]
are all equivalences of symmetric monoidal triangulated categories.  Furthermore, the composite and intermediate sheafifications commute on compact objects with the six functors by \cite{CisDeg} theorem 4.29.  Finally, for any $k$-variety $X$ with structure morphism $f:X\to\Spec k$, one has that $f_*f^!(\mathbf1_k)\cong M^c(X)$ in $\dm_{cdh}(k, \Q)$ via \cite{CisDeg2} proposition 8.10.  Since the image of $f_*f^!(\mathbf1_k)$ considered in $\dm_B(k)$ and $M^c(X)$ considered in $\dm(k, \Q)$ coincide in $\dm_{cdh}(k, \Q)$, one must have that $f_*f^!(\mathbf1_k)$ maps under sheafification to an object isomorphic to $M^c(X)$ in $\dm(k, \Q)$.
\end{proof}

\begin{thm}
Considering a perfect base field $k$ and rational coefficients, the map 
\[
K(M^c_k):K(\Var_k)\to K(\DM_B^c(k))
\]
yields a derived lift of the Gillet-Soul\'e motivic measure.
\end{thm}
\begin{proof}
Recall that over rational coefficients, for Bondarko's weight map on the level of triangulated categories $t_\Q:DM_{gm}(k, \Q)\to K^\flat\text{Chow}(k, \Q)$ maps the compactly supported motive $M^c_k(X)$ to the weight complex $W(X)$ for any finite type $k$-scheme $X$.  Furthermore, $t_\Q$ induces an isomorophism on the level of Grothendieck Groups.  Noting that under the isomorphism $K_0(K^\flat\text{Chow}(k, \Q))\cong K_0(\text{Chow}(k, \Q))$, the image of the Gillet-Soul\'e motivic measure is precisely $[W(X)]$.  Thus, further noting that the Grothendieck group of a stable $\infty$-category is the same as the grothendieck group of its (triangulated) homotopy category, stringing all of our equivalences together, we have that the Gillet-Soul\'e motivic measure can be factored as
\[
K_0(\Var_k)\overset{K_0(M^c_k)}{\to}K_0(\DM_B^c(k))\cong K_0(DM_{gm}(k, \Q))\cong K_0(\text{Hot}^\flat\text{Chow}(k, \Q))\cong K_0(\text{Chow}(k, \Q)),
\]
which shows that $K(M^c_k):K(\Var_k)\to K(\DM_B^c(k))$ is a derived lift of the Gillet-Soul\'e motivic measure.
\end{proof}

\begin{rem}
Let us briefly work in the model of quasicategories and suppose the existence of a motivic $t$-structure on $\dm_{gm}(k, \Q)$ \cite{Bei} .  This, in particular, may be lifted to a $t$-structure on $\DM_B^c(k)$ via the above equivalence.  Now, Barwick's Theorem of the Heart states that for any bounded $t$-structure on a stable $\infty$-category $\cA$, one necessarily obtains an equivalence $K(\cA)\simeq K(\cA^\heartsuit)$, where the latter K-theory is taken as an exact $\infty$-category.  This exact K-theory coincides with the classical K-theory of an abelian category if $\cA^\heartsuit$ is actually just the nerve of an abelian category.  Thus, if one can show that the properties of the hypothetical motivic $t$-structure on $\dm_{gm}(k, \Q)$ imply that the lift to $\DM_B^c(k)$ is accessible (in the sense of Lurie \cite{Lur2} definition 1.4.4.12) and bounded upon lifting (it is known to be bounded on the triangulated homotopy category), then $K(\DM_B^c(k))$ must in fact be the K-theory of the hypothetical abelian category of mixed motives.
\end{rem}

\subsection{Some Other Approaches}

In addition to the machinery which we have built above, there are several other potential approaches that one can take to the problem of constructing a derived lift to the Gillet-Soul\'e motivic measure.  One of them, using the language of cofibration categories and topological triangulated categories, \cite{Schw} , is very likely equivalent to the approach above.  This approach is similar to one developed independently in the excellent paper of Braunling, Groechenig, and Nanavaty \cite{BGN}, although their approach makes use of a model for the K-theory of DG categories more akin to that defined above for stable categories. As such, it also seems likely equivalent. There is another possible approach one could take making use of recent work by Cisinski and Bunke \cite{BunCis} on the K-theory of additive categories which is almost surely not equivalent to any of the above.  We will discuss the former method, but forego further discussion of the latter two, referring the interested reader to \cite{BGN} and \cite{BunCis}.

Given any pretriangulated DG-category $\cC$, one can construct a cofibration category (which is, in this case, a Waldhausen category) lying between it and $\textbf{Ho}(\cC)$ as follows.  One defines the \emph{cycle category} of $\cC$, denoted $\cZ(\cC)$, to be the additive category with the same objects as $\cC$, but for any $X, Y\in\text{Ob}(\cC)$, $\Hom_{\cZ(\cC)}(X, Y)=\Ker(\Hom_\cC(X, Y)^0\overset{d}{\to}\Hom_\cC(X, Y)^1)$ (the "closed morphisms" of degree $0$ in $\Hom_\cC(X, Y)$) with the composition induced by that of $\cC$.

$\cZ(\cC)$ can be made into a stable cofibration category (which happens to be a Waldhausen category) in the following way:
\begin{itemize}
    \item Weak equivalences are those morphisms which become equivalences in $\Ho(\cC)$
    \item Cofibrations are those morphisms $f\in\Hom_{\cZ(\cC)}(X, Y)$ for $X, Y\in\text{Ob}(\cC)$ such that for any $Z\in\Ob(\cC)$, one has that the induced morphism
    \[
    \Hom_\cC(f, Z):\Hom_\cC(Y, Z)\to\Hom_\cC(X, Z)
    \]
    is a surjection
\end{itemize}
The proof that $\cZ(\cC)$ equipped with these weak equivalences and cofibrations is a stable cofibration category is \cite{Schw} proposition 3.2.

Recall further from \cite{Schw} remark 1.3 that if $\cC$ is a stable cofibration category, a triangle $X\to Y\to Z\to X[1]$ in $\Ho(\cC)$ is distinguished if and only if it is connected by a zigzag of weak equivalences to a cofiber sequence in $\cC$.  This will be used in the following theorem.

\begin{thm}
There is a weakly W-exact functor $M^c:\Var_S\to\cZ(\cD\cM_{gm}(S, R))$ sending each variety (smooth or otherwise) $X\in\Var_S$ to its compactly supported motive $M^c(X)$.
\end{thm}

\begin{proof}
This proof will proceed in several stages.  First, we need to construct the three functors that will form the basis of the rest of the proof:
\begin{itemize}
    \item Recall that the cofibrations in the SW-category $\Var_S$ are merely the closed immersions.  In particular, closed immersions are proper, so for any closed immersion $f:X\hookrightarrow Y$ in $\Var_S$, one obtains the map $f_*:M^c(X)\to M^c(Y)$ functorially.  This yields the data of a functor $\cof(\Var_S)\to\cZ(\cD\cM_{gm}(S, R))$
    \item Recall that complement maps in the SW-category $\Var_S$ are the open immersions.  Since open immersions are smooth, for any open immersion $f:X\overset{\circ}{\to}Y$, one obtains the map $f^*:M^c(Y)\to M^c(X)$ contravariantly and functorially.  This yields the data of a functor $\comp(\Var_S)^c\to\cZ(\cD\cM_{gm}(S, R))$
    \item Since weak equivalences in the SW-category $\Var_S$ are merely isomorphisms, and are thus closed immersions, the lower star construction yields a functorial mapping into $\cZ(\cD\cM_{gm}(S, R))$.  This map is furthermore in $w\cZ(\cD\cM_{gm}(S, R))$ in this case because isomorphisms map functorially to isomorphisms, and all isomorphisms are in $w\cZ(\cD\cM_{gm}(S, R))$
\end{itemize}
Now we can merely verify that all the appropriate properties are satisfied.
\begin{itemize}
    \item This is an application of \cite{SusVoe} proposition 3.6.5.  In particular, given any cartesian square
    \[
    \begin{tikzcd}
 X\arrow[d, "i"]\arrow[r, "j"] & Z \arrow[d, "i'"]\\
   Y \arrow[r, "j'"]& W
\end{tikzcd},
    \]
    of schemes with horizontal arrows closed immersions and vertical arrows open immersions, one obtains a commutative diagram
    \[
    \begin{tikzcd}
 R^c[X]\arrow[r, hook, "j_*"] & R^c[Z] \\
  R^c[Y]\arrow[u, "i^*"]\arrow[r, hook, "j'_*"]& R^c[W]\arrow[u, "i'^*"]
\end{tikzcd}
    \]
    on the level of Nisnevich sheaves with transfers and hence a commutative diagram
    \[
    \begin{tikzcd}
 M^c(X)\arrow[r, hook, "j_*"] & M^c(Z) \\
  M^c(Y)\arrow[u, "i^*"]\arrow[r, hook, "j'_*"]& M^c(W)\arrow[u, "i'^*"]
\end{tikzcd}
    \]
    on the level of motives
    \item  As was proven in \cite{BeiVol} (see equation 6.9.1), given a closed immersion $i:Z\hookrightarrow X$ with complementary open immersion $j: U=X-Z\overset{\circ}{\to}X$, one has a distinguished triangle $M^c(Z)\overset{i_*}{\to}M^c(X)\overset{j^*}{\to}M^c(U)\to M^c(Z)[1]$ in $DM_{gm}(S, R)$.  By \cite{Schw} remark 1.3, that this triangle is distinguished in the homotopy category implies that $M^c(Z)\overset{i_*}{\to}M^c(X)\overset{j^*}{\to}M^c(U)$ is weakly equivalent via a zigzag of weak equivalences to a cofiber sequence.  In other words, noting that $M^c(\emptyset)\simeq0$, we have a homotopy cocartesian square
    \[
    \begin{tikzcd}
 M^c(Z)\arrow[d]\arrow[r, "i_*"] & M^c(X) \arrow[d, "j^*"]\\
   M^c(\emptyset) \arrow[r]& M^c(U)
\end{tikzcd}
    \]
    \item  Given any commutative diagram of schemes 
    \[
    \begin{tikzcd}
 X\arrow[d, "g"]\arrow[r, hook, "f"] & Z \arrow[d, "g'"]\\
  Y \arrow[r, hook, "f'"]& W
\end{tikzcd}
    \]
    
    with vertical morphisms isomorphisms and horizontal morphisms closed immersions, noting that the direct image is a functorial assignment, one gets a commutative square
    \[
    \begin{tikzcd}
 M^c(X)\arrow[d, "g_*"]\arrow[r, hook, "f_*"] & M^c(Z) \arrow[d, "g'_*"]\\
  M^c(Y) \arrow[r, hook, "f'_*"]& M^c(W)
\end{tikzcd}.
    \]
    If one instead assumes that the horizontal morphisms are open immersions, one may note that any commutative square with vertical morphisms given by isomorphisms is cartesian, and the result then follows from the result for the first bullet point.
    
\end{itemize}
\end{proof}

\section{Conclusion}
The procedure we have outlined above allows us to go from abstract six functors formalisms to derived motivic measures.  Now that we have the method, and an application of it, there are many more questions that need answering, and many more possible directions in which to take this work.

First and foremost, one might seek to upgrade many of the other categories defined by Cisinski and D\'eglise and upgrade as well the many premotivic adjunctions we have already discussed on the level of triangulated categories to their respective $\infty$-categorical analogues.

In recent work, Khan and Ravi extended Khan's work discussed above to more general categories of (derived) algebraic stacks (see \cite{KhaRav} and \cite{Khan3}).  This opens up exciting new avenues for the work we undertake here.  In addition to potentially allowing for exploration of the K-theory of algebraic stacks via derived motivic measures of the type discussed above, it opens up the possibility of looking at $G$-equivariant motivic measures.  Among other things, this would allow one to further generalize aspects of the work found in \cite{LMM}.

Another immediate direction in which to take this work would be to compare it to that found in \cite{CWZ} .  In it, Campbell, Wolfson, and Zakharevich define a spectral lifting of the Hasse-Weil zeta function making use of compactly supported $l$-adic cohomology.  One has reason to strongly suspect that, using $l$-adic sheaves as our target, one might obtain the same map of K-theory spectra up to homotopy.

More generally, this approach seems to allow one to spectrally lift many of the classical motivic measures, especially those related to categories of sheaves.  All of this requires further investigation.  This will perhaps be carried out in a later version of this paper, or in a sequel.

\section{Acknowledgements}

The author would like to thank Matilde Marcolli for many vital and engaging discussions during the formation of this paper.  In addition, they would like to thank Emily Riehl for a very valuable email exchange regarding several of the categorical constructions necessary for the paper and Justin Campbell for a discussion of equivariant categories of sheaves that will likely form the basis for further work.  Finally, the author would like to thank Konrad Pilch and Victor Zhang for comments on an earlier draft of this paper.

\end{document}